\documentclass[12pt]{amsart}
\usepackage{amsmath}
\usepackage{amstext}
\usepackage{amsfonts}
\usepackage{amssymb}
\usepackage{amsthm}
\usepackage{amsrefs}
\usepackage{color}
\usepackage{hyperref}
\usepackage{microtype}
\usepackage{tikz-cd}
\usepackage{hyphenat}

\allowdisplaybreaks[3]

\definecolor{darkblue}{rgb}{0.0,0.0,0.3}
\hypersetup{colorlinks,breaklinks,linkcolor=red,urlcolor=red,anchorcolor=red,citecolor=red}

\theoremstyle{plain}
\newtheorem{thm}{Theorem}[section]

\newtheorem{cor}[thm]{Corollary}
\newtheorem{prop}[thm]{Proposition}
\newtheorem{lem}[thm]{Lemma}

\theoremstyle{definition}
\newtheorem{defn}[thm]{Definition}
\newtheorem{example}[thm]{Example}
\newtheorem{rem}[thm]{Remark}

\newcommand{\bC}{{\mathbb{C}}}

\newcommand{\bF}{{\mathbb{F}}}

\newcommand{\bN}{{\mathbb{N}}}

\newcommand{\bR}{{\mathbb{R}}}

\newcommand{\A}{{\mathcal{A}}}
\newcommand{\B}{{\mathcal{B}}}

\newcommand{\cM}{{\mathcal{M}}}

\renewcommand{\P}{{\mathcal{P}}}

\newcommand{\U}{{\mathcal{U}}}

\newcommand{\rA}{{\mathrm{A}}}
\newcommand{\rB}{{\mathrm{B}}}

\newcommand{\rC}{{\mathrm{C}}}
\newcommand{\rL}{{\mathrm{L}}}
\newcommand{\rP}{{\mathrm{P}}}

\newcommand{\bP}{{\mathbf{P}}}
\newcommand{\bX}{{\mathbf{X}}}

\renewcommand{\phi}{\varphi}

\newcommand{\ca}{\mathrm{C}^*}

\newcommand{\id}{\operatorname{id}}

\newcommand{\ol}{\overline}
\newcommand{\ran}{\operatorname{ran}}

\newcommand{\spn}{\operatorname{span}}

\newcommand{\qfor}{\quad\text{for}\quad}

\DeclareMathOperator{\cbmaps}{\text{CB}}

\DeclareMathOperator{\ncpmaps}{CP_{nor}}
\DeclareMathOperator{\ucpmaps}{\text{UCP}}
\DeclareMathOperator{\bor}{Bor}

\DeclareMathOperator{\Irr}{Irr}
\DeclareMathOperator{\Irrep}{Irrep}
\DeclareMathOperator{\Rep}{Rep}

\newcommand{\cmin}{\mathrm{C}_{\textup{min}}^*}
\newcommand{\cmax}{\mathrm{C}_{\textup{max}}^*}
\newcommand{\tmin}{\otimes_{\operatorname{min}}}
\newcommand{\tcomm}{\otimes_{\operatorname{c}}}
\newcommand{\tmax}{\otimes_{\operatorname{max}}}

\DeclareMathOperator*{\maxt}{\otimes_{\mathrm{max}}}

\begin{document}
\title{Noncommutative Choquet simplices}

\author[M. Kennedy]{Matthew Kennedy}
\address{Department of Pure Mathematics\\ University of Waterloo\\
Waterloo, Ontario \; N2L 3G1 \\ Canada}
\email{matt.kennedy@uwaterloo.ca}

\author[E. Shamovich]{Eli Shamovich}
\address{Department of Mathematics\\ Ben-Gurion University of the Negev\\
Beer-Sheva \; 8410501 \\ Israel} 
\email{shamovic@bgu.ac.il}

\begin{abstract}
We introduce a notion of noncommutative Choquet simplex, or briefly an nc simplex, that generalizes the classical notion of a simplex. While every simplex is an nc simplex, there are many more nc simplices. They arise naturally from C*-algebras and in noncommutative dynamics. We characterize nc simplices in terms of their geometry and in terms of structural properties of their corresponding operator systems.

There is a natural definition of nc Bauer simplex that generalizes the classical definition of a Bauer simplex. We show that a compact nc convex set is an nc Bauer simplex if and only if it is affinely homeomorphic to the nc state space of a unital C*-algebra, generalizing a classical result of Bauer for unital commutative C*-algebras.

We obtain several applications to noncommutative dynamics. We show that the set of nc states of a C*-algebra that are invariant with respect to the action of a discrete group is an nc simplex. From this, we obtain a noncommutative ergodic decomposition theorem with uniqueness.

Finally, we establish a new characterization of discrete groups with Kazhdan's property (T) that extends a result of Glasner and Weiss. Specifically, we show that a discrete group has property (T) if and only if for every action of the group on a unital C*-algebra, the set of invariant states is affinely homeomorphic to the state space of a unital C*-algebra.
\end{abstract}

\subjclass[2010]{Primary 46A55, 46L07, 47A20; Secondary 46L52, 47L25}
\keywords{noncommutative convexity, noncommutative Choquet simplex, C*-algebra, C*-dynamical system, dynamical system, ergodic theory, Kazhdan's property (T)}
\thanks{First author supported by NSERC Grant Number 50503-10787.}
\maketitle

\setcounter{tocdepth}{1}
\tableofcontents

\section{Introduction}

A Choquet simplex, or more briefly a simplex, is a special kind of convex set that plays an important role in the theory of convexity. The key property of a simplex is that every point can be uniquely represented as the barycenter of a probability measure supported on the extreme boundary of the simplex. This fact underlies many important decomposition results in classical analysis.

An important example of a simplex is the set of tracial states of a unital C*-algebra. If the C*-algebra is commutative, then this simplex is the entire state space, and the extreme points are precisely the characters. A classical result of Bauer (see e.g. \cite{Alf1971}*{Theorem II.4.3}) asserts that a compact convex set is affinely homeomorphic to the space of probability measures on a compact space, and hence to the state space of a unital commutative C*-algebra, if and only if it is a Bauer simplex, meaning that it is a simplex with closed extreme boundary.

Another important example of a simplex due to Feldman \cite{Fel1963} is the set of probability measures invariant with respect to the action of a locally compact group on a compact space. The extreme points of this simplex are the ergodic measures, and so this fact, combined with the Choquet-Bishop-de Leeuw integral representation theorem, implies the ergodic decomposition theorem, i.e. that every invariant probability measure can be uniquely represented as the barycenter of a probability measure supported on the ergodic measures (see e.g. \cite{Phe2001}*{Chapter 12}).

Glasner and Weiss \cite{GW1997} obtained a remarkable dynamical characterization of groups with Kazhdan's property (T) that refines Feldman's result. Specifically, they proved that a locally compact group has property (T) if and only if every simplex arising in this way from an action of the group on a compact space is a Bauer simplex.

Recently, Davidson and the first author \cite{DK2019} introduced a theory of noncommutative convexity, along with a corresponding noncommutative Choquet theory, for the study of ``higher-order'' convex structure that occurs in noncommutative mathematics. This theory has its roots in Arveson's boundary theory for operator algebras \cites{Arv1969,Arv2008,DK2015}, where examples of higher-order convex structure are explicitly identified for the first time, and in Wittstock's theory of matrix convexity \cites{Wit1981,EW1997,WW1999}. 

In this paper, we will work within the framework of noncommutative convexity and introduce a special kind of noncommutative convex set that we call a noncommutative Choquet simplex, or more briefly an nc simplex, since it generalizes the classical notion of a simplex. We will demonstrate the utility of this notion by establishing noncommutative generalizations of the classical results mentioned at the beginning of this introduction.

The central objects of interest in noncommutative convexity are the nc convex sets. A fundamental result is the duality theorem from \cite{DK2019}*{Section 3}, which asserts that every compact nc convex set is affinely homeomorphic to the nc state space of an operator system, which is a closed unital self-adjoint subspace of a C*-algebra. Many interesting examples arise in this way. However, interesting examples of compact nc convex sets have recently been found to arise geometrically from the theory of linear matrix inequalities (see e.g. \cites{HM2012,HMV2006}) and the theory of noncommutative real algebraic geometry (see e.g. \cites{DDSS2017, HKM2013, HKMS2019}).

An nc convex set over an operator space $E$ is a graded set $K = \coprod K_n$, where each graded component $K_n$ consists of $n \times n$ matrices over $E$, and the graded components are related by requiring that $K$ is closed under taking direct sums and compressions by isometries. The union is taken over all $n \leq \kappa$ for a sufficiently large infinite cardinal number $\kappa$. In the separable setting, it typically suffices to take $\kappa = \aleph_0$.

If $E$ is a dual operator space, then there is a natural topology on each graded component $K_n$, and $K$ is an nc convex set if and only if it is closed under nc convex combinations, meaning that $\sum \alpha_i^* x_i \alpha_i \in K_n$ for every bounded family $\{x_i \in K_{n_i}\}$ and every family $\{\alpha_i \in \cM_{n_i,n}\}$, such that $\sum \alpha_i^* \alpha_i = 1_n$. Here, $\cM_{n_i,n}$ denotes the set of $n_i \times n$ scalar matrices with uniformly bounded finite submatrices. In this case, $K$ is said to be compact if each $K_n$ is compact.

The duality between compact nc convex sets and operator systems is implemented by the functor that maps a compact nc convex set $K$ to the operator system $\rA(K)$ of continuous affine nc functions on $K$. The inverse functor maps an operator system $S$ to its nc state space $K = \coprod K_n$, where $K_n$ is the set of unital complete order homomorphisms from $S$ into the space $\cM_n$ of $n \times n$ scalar matrices with uniformly bounded finite submatrices.

A large portion of \cite{DK2019} relates to the development of a noncommutative Choquet theory that generalizes much of classical Choquet theory. A key idea is the notion of a representing map for a point in a compact nc convex set, which is a noncommutative analogue of a representing probability measure for a point in a compact convex set. There is a partial order on the family of representing maps called the nc Choquet order, which is a noncommutative analogue of the Choquet order on the family of representing probability measures. A compact convex set is said to be a simplex if there is a unique representing probability measure that is maximal in the Choquet order. Motivated by this definition, a compact nc convex set is said to be an nc simplex if every point in the set has a unique representing map that is maximal in the nc Choquet order.

Every simplex is an nc simplex, but we will see that there are many more nc simplices than there are simplices. In addition to establishing their basic structural properties, we will give several characterizations of nc simplices in terms of the structure of the corresponding operator system of continuous affine nc functions.

As mentioned above, we will apply our results to obtain noncommutative generalizations of the classical results mentioned at the beginning of this introduction. In each case, we will see that it is essential to consider the noncommutative convex structure of the objects under consideration, since considering only the convex structure leads to unsatisfactory results. 

For example, in contrast to Bauer's characterization of state spaces of unital commutative C*-algebras, the state space of a unital noncommutative C*-algebra is never a simplex. While a characterization of state spaces of arbitrary C*-algebras was obtained in deep work of Alfsen and Shultz \cite{AS2003}, their result has a much different flavour than Bauer's.

It turns out that there is a natural definition of an nc Bauer simplex that generalizes the classical definition of a Bauer simplex. By working with the nc state space of a C*-algebra instead of the state space, we obtain a very natural noncommutative generalization of Bauer's result.

\begin{thm}
Let $K$ be a compact nc convex set. Then $K$ is affinely homeomorphic to the nc state space of a unital C*-algebra if and only if $K$ is an nc Bauer simplex.
\end{thm}

Feldman's result that the set of probability measures invariant with respect to the action of a group on a compact space is a simplex is a key result in ergodic theory. It is equivalent to the assertion that the set of states invariant with respect to the action of a group on a unital commutative C*-algebra is a simplex.

In contrast, the set of states invariant with respect to the action of a group on a unital noncommutative C*-algebra is generally not a simplex. As a result, there is no general ergodic decomposition-type theorem with uniqueness over the set of invariant states (see e.g. \cite{LR1967}). This is an issue that has received a great deal of attention in, for example, the C*-algebraic approach to statistical mechanics and quantum field theory (see e.g. \cite{BR1981}*{Chapter 1} and \cite{Rue1999}*{Chapter 6}).

By once again working with the nc state space of a C*-algebra instead of the state space, we obtain a very natural generalization of Feldman's result about invariant probability measures.

\begin{thm}
Let $\Gamma$ be a discrete group acting on a unital C*-algebra $A$ with nc state space $K$. Then the set $K^\Gamma$ of invariant nc states is an nc simplex.
\end{thm}

An invariant state is said to be ergodic if it is an extreme point of the set of invariant nc states. As a consequence of the above theorem and the noncommutative Choquet-Bishop-de Leeuw theorem from \cite{DK2019}*{Theorem 9.2.3}, we obtain a noncommutative ergodic decomposition theorem with uniqueness.

\begin{thm}[Noncommutative ergodic decomposition] \strut \\
Let $\Gamma$ be a discrete group acting on a C*-algebra $A$ with nc state space $K$. For an invariant nc state $\mu \in K^\Gamma$, there is a unique nc state $\nu:\rC(K^\Gamma) \to \cM_n$ that is maximal in the nc Choquet order and represents $\mu$. Moreover, $\nu$ is supported on the set $\partial (K^\Gamma)$ of ergodic nc states.
\end{thm}

In the separable case, as a consequence of the integral representation theorem from \cite{DK2019}*{Theorem 10.3.11}, we obtain a noncommutative integral decomposition theorem for invariant nc states in terms of ergodic nc states that more closely resembles the classical ergodic decomposition theorem. We refer to Section \ref{sec:nc-ergodic-theory} for the definition of integration against an admissible nc probability measure.

\begin{thm}
Let $\Gamma$ be a discrete group acting on a separable C*-algebra $A$ with nc state space $K$. For an invariant nc state $\mu \in K^\Gamma$ there is an admissible nc probability measure $\lambda$ on $K$ that represents $\mu$ and is supported on the set $\partial (K^\Gamma)$ of ergodic nc states, meaning that
\[
\mu(a) = \int_K a\ d\lambda, \qfor a \in \A.
\]
\end{thm}

Finally, the result of Glasner and Weiss characterizing groups with Kazhdan's property (T) is equivalent to the assertion that a group has property (T) if and only if for every action of the group on a unital commutative C*-algebra, the set of invariant states is affinely homeomorphic to the state space of a unital commutative C*-algebra. We obtain a noncommutative extension of this result for discrete groups.

\begin{thm}
A discrete group $\Gamma$ has Kazhdan's property (T) if and only if for every action of $\Gamma$ on a unital C*-algebra, the set of invariant states is affinely homeomorphic to the state space of a unital C*-algebra.
\end{thm}

\section*{Acknowledgements}

The first author is grateful to Martino Lupini for enlightening conversations about the classical Poulsen simplex and the Kirchberg\hyp{}Wassermann operator system.

\section{Classical Choquet simplices}

In this section, we will briefly review the definition of a classical Choquet simplex, or briefly a simplex, as well as the requisite background in Choquet theory. For references on Choquet theory and simplices, we direct the reader to the books \cite{Alf1971}, \cite{LMNS2010} and \cite{Phe2001}.

A simplex can be viewed as a higher-dimensional generalization of the notion of a triangle. The definition is motivated by the basic fact from convex geometry that a compact convex subset $C \subseteq \bR^2$ is a triangle if and only if every point $x \in C$ can be expressed uniquely as a (finite) convex combination of points in the extreme boundary $\partial C$ of $C$, i.e. the set of extreme points of $C$. Before recalling the definition of a simplex, we briefly recall the Choquet order and the integral representation theorem of Choquet-Bishop-de Leeuw.

Let $V$ be a locally convex vector space over the real or complex numbers and let $C \subseteq V$ be a compact convex subset. For probability measures $\mu$ and $\nu$ on $C$, $\nu$ is said to dominate $\mu$ in the Choquet order, written $\mu \prec \nu$, if $\mu(f) \leq \nu(f)$ for every convex function $f \in \rC(C)$. Here we have written $\mu(f)$ for the integral $\int_C f \, d\mu$ of $f$ against $\mu$. The Choquet order is a partial order on the space of probability measures on $C$. Heuristically, if $\mu \prec \nu$, then the support of $\mu$ is further than the support of $\nu$ to the extreme boundary $\partial C$. 

If $C$ is metrizable, then the extreme boundary $\partial C$ is Borel, and the maximality of $\mu$ in the Choquet order implies that $\mu$ is supported on $\partial C$ in the usual sense that $\mu(C \setminus \partial C) = 0$. On the other hand, if $C$ is non-metrizable, then $\partial C$ is not necessarily Borel. However, the maximality of $\mu$ in the Choquet order still implies that $\mu$ is supported on $\partial C$ in an appropriate sense. Specifically, it implies that $\mu(X) = 0$ for every Baire set $X \subseteq C \setminus \partial C$.

The integral representation theorem of Choquet-Bishop-de Leeuw asserts that every point $x \in C$ is the barycenter of a probability measure $\mu$ on $C$ that is maximal in the Choquet order. From above, this implies that $\mu$ is supported on the extreme boundary $\partial C$ in an appropriate sense.

\begin{defn}
A compact convex set $C$ is a {\em simplex}, if for every point $x \in C$ there is a unique probability measure $\mu$ on $C$ that is maximal in the Choquet order.
\end{defn}

If $C$ is metrizable, then a probability measure $\mu$ on $C$ is maximal in the Choquet order if and only if it is supported on the extreme boundary $\partial C$. However, if $C$ is non-metrizable, then this is no longer true. In fact, pathological examples exist of non-metrizable compact convex sets that should, by every right, be called simplices, with the property that there are multiple probability measures supported on the extreme boundary with the same barycenter (see e.g. \cite{Phe2001}*{Section 10}). For this reason, simplices are defined in terms of the Choquet order.

For a compact convex set $C$, let $\rA(C)$ denote the space of continuous affine functions on $C$. Then $\rA(C)$ is a function system, i.e. a unital self-adjoint subspace of the commutative C*-algebra $\rC(C)$. In fact, the functor mapping $C$ to $\rA(C)$ implements the dual equivalence between the category of compact convex sets with morphisms consisting of continuous affine maps and the category of function systems with morphisms consisting of unital order homomorphisms (see \cite{DK2019}*{Section 3.1}).

\section{Noncommutative convexity and noncommutative Choquet theory} \label{sec:nc-convexity}

\subsection{Noncommutative convex sets}

Throughout this paper, we will work within the framework of noncommutative convexity introduced in \cite{DK2019}. In this section, we briefly review some of the relevant material. For a reference on operator spaces and operator systems, we direct the reader to the books \cite{Pau2002} and \cite{Pis2003}. 

For an operator space $E$ and nonzero cardinal numbers $m$ and $n$, let $\cM_{m,n}(E)$ denote the operator space consisting of $m \times n$ matrices over $E$ with uniformly bounded submatrices.  For the special case when $E = \bC$, let $\cM_{m,n} = \cM_{m,n}(\bC)$. Let $\cM_n(E) = \cM_{n,n}(E)$ and $\cM_n = \cM_n(\bC)$. 

For each nonzero cardinal number $n$, we fix a Hilbert space $H_n$ of dimension $n$ and identify $\cM_{m,n}$ with the space $\B(H_n,H_m)$ of bounded operators from $H_n$ to $H_m$. 

Let $m,n,p$ be nonzero cardinal numbers. For $x \in \cM_n(E)$ and scalar matrices $\alpha \in \cM_{m,n}$ and $\beta \in \cM_{n,p}$, the product $\alpha x \beta \in \cM_{m,p}(E)$ can be defined as compositions under appropriate operator space embeddings. 

For a fixed (sufficiently large) nonzero cardinal number $\kappa$, let $\cM(E) = \coprod \cM_n(E)$, where the disjoint union is taken over all nonzero cardinal numbers $n \leq \kappa$. As above, for the special case $E = \bC$, let $\cM = \cM(\bC)$. The choice of $\kappa$ will depend on the setting. For example, if we are considering separable operator spaces, then it will typically suffice to take $\kappa = \aleph_0$. We will work with the understanding that $\kappa$ exists and simply write e.g. ``for all $n$'' instead of ``for all $n \leq \kappa$.''

If $E$ is a dual operator space with a distinguished predual $E_*$, then we identify $\cM_n(E)$ with the space $\cbmaps(E_*,\cM_n)$ of completely bounded maps from $E_*$ to $\cM_n$ under the natural operator space isomorphism. We equip $\cM_n(E)$ with the corresponding point-weak* topology. For $\cM_n$, this is the usual weak* topology.

The next definition is \cite{DK2019}*{Definition 2.2.1}.

\begin{defn}
An {\em nc convex set} over an operator space $E$ is a graded subset $K = \coprod K_n \subseteq \cM(E)$ that is closed under direct sums and compressions, meaning that
\begin{enumerate}
\item $\sum \alpha_i x_i \alpha_i^* \in K_n$ for every family $\{x_i \in K_{n_i}\}$ and every family of isometries $\{\alpha_i \in \cM_{n,n_i} \}$ satisfying $\sum \alpha_i \alpha_i^* = 1_n$,
\item $\beta^* x \beta \in K_m$ for every $x \in K_n$ and every isometry $\beta \in \cM_{n,m}$. 
\end{enumerate}
The set $K$ is said to be {\em compact} if $E$ is a dual operator space and each $K_n$ is compact in the topology on $\cM_n(E)$. 
\end{defn}

It is often convenient to work with the heuristic that a set is an nc convex set if it contains all noncommutative convex combinations of its elements. This can be made precise for nc convex sets over dual operator spaces. We require the following definition from \cite{DK2019}*{Definition 2.2.7}.

\begin{defn}
Let $K$ be an nc convex set. An {\em nc convex combination} of elements in $K$ is an expression of the form $\sum \alpha_i^* x_i \alpha_i$ for $n$, a bounded family $\{x_i \in K_{n_i}\}$ and a family $\{\alpha_i \in \cM_{n_i,n} \}$ satisfying $\sum \alpha_i^* \alpha_i = 1_n$.
\end{defn}

The next result is \cite{DK2019}*{Proposition 2.2.8}.

\begin{prop}
Let $E$ be a dual operator space. A subset $K \subseteq \cM(E)$ is an nc convex set if and only if it is closed under nc convex combinations.
\end{prop}

\begin{example}
Let $S$ be an operator system. An {\em nc state} on $S$ is a unital completely positive map of the form $\phi : S \to \cM_n$ for some $n$. We let $\ucpmaps(S,\cM_n)$ denote the collection of all nc states with range in $\cM_n$. The {\em nc state space} of $S$ is the nc convex set $K = \coprod K_n$ defined by
\[
K_n = \ucpmaps(S,\cM_n).
\]
Identifying each $K_n = \ucpmaps(S,\cM_n)$ with a subspace of $\cM_n(S^*)$, we see that $K$ is an nc convex set over the dual operator space $S^*$. Moreover, each $K_n$ is compact in the corresponding point-weak* topology, so $K$ is compact. We will see shortly that every compact nc convex set arises in this way as the nc state space of an operator system.
\end{example}

We will require the notion of a dilation of a point in a compact nc convex set, as well as the corresponding notion of a maximal point. The next definition is \cite{DK2019}*{Definition 5.1.1}.

\begin{defn}
Let $K$ be a compact nc convex set. A point $x \in K_m$ is {\em dilated} by a point $y \in K_n$ if there is an isometry $\alpha \in \cM_{n,m}$ such that $x = \alpha^* y \alpha$. If $y \simeq x \oplus z$ for some $z \in K$, where the decomposition is taken with respect to $\alpha$, then the dilation is said to be trivial. The point $x$ is said to be {\em maximal} if it has no non-trivial dilations.
\end{defn}

It was shown in \cite{DK2019} that there are appropriate noncommutative analogues of classical results like the Krein-Milman theorem and Milman's partial converse to the Krein-Milman theorem for extreme points in a compact convex set. We will require the notion of an extreme point of a compact nc convex set. The next definition is \cite{DK2019}*{Definition 6.1.1}.

\begin{defn}
Let $K$ be a compact nc convex set. A point $x \in K_n$ is {\em extreme} if whenever $x$ is written as a finite nc convex combination $x = \sum \alpha_i^* x_i \alpha_i$ for $\{x_i \in K_{n_i} \}$ and $\{\alpha_i \in \cM_{n_i,n} \}$ satisfying $\sum \alpha_i^* \alpha_i = 1_n$, then each $\alpha_i^* x_i \alpha_i$ is a scalar multiple of $x$ and each $x_i$ decomposes with respect to the range of $\alpha_i$ as a direct sum $x_i = y_i \oplus z_i$ for $y_i,z_i \in K$ with $y_i$ unitarily equivalent to $x$. The {\em extreme boundary} $\partial K = \coprod (\partial K)_n$ of $K$ is the set of all extreme points.
\end{defn}

\subsection{Noncommutative functions}

The next definition is Definition \cite{DK2019}*{Definition 4.2.1}.

\begin{defn} \label{defn:nc_func}
Let $K$ be a compact nc convex set. A function $f~:~K \to \cM$ is an {\em nc function} if it is graded, respects direct sums and is unitarily equivariant, meaning that
\begin{enumerate}
\item $f(K_n) \subseteq K_n$ for all $n$,
\item $f(\sum \alpha_i x_i \alpha_i^*) = \sum \alpha_i f(x_i) \alpha_i^*$ for every family $\{x_i \in K_{n_i} \}$ and every family of isometries $\{\alpha_i \in \cM_{n_i,n}\}$ satisfying $\sum \alpha_i \alpha_i^* = 1_n$,
\item $f(\beta^* x \beta) = \beta f(x) \beta^*$ for every $x \in K_n$ and every unitary $\beta \in \cM_n$.
\end{enumerate}
The function $f$ is an {\em affine nc function} if it is an nc function that is, in addition, equivariant with respect to isometries, meaning that
\begin{enumerate}
\item[(3')] $f(\gamma^* x \gamma) = \gamma^* f(x) \gamma$ for every $x \in K_m$ and every isometry $\gamma \in \cM_{m,n}$.
\end{enumerate}
An nc function $f$ is {\em bounded} if $\|f\|_\infty < \infty$, where $\|f\|_\infty$ denotes the uniform norm
\[
\|f\|_\infty = \sup_{x \in K} \|f(x)\|.
\]
Finally, an nc function $f$ is {\em continuous} if for every $n$, the restriction $f|_{K_n} \colon K_n \to \cM_n$ is continuous.
\end{defn}

Let $\rA(K)$ denote the operator system of continuous affine nc functions on $K$, let $\rC(K)$ denote the space of continuous nc functions on $K$ and let $\rB(K)$ denote the space of bounded nc functions on $K$. Note that $\rA(K) \subseteq \rC(K) \subseteq \rB(K)$. For $f \in \rB(K)$ the adjoint $f^*$ is a bounded nc function defined by $f^*(x) = f(x)^*$ for $x \in K$. The spaces $\rC(K)$ and $\rB(K)$ are C*-algebras with respect to the adjoint, the pointwise product and the uniform norm.

It follows from \cite{DK2019}*{Theorem 4.4.3} and \cite{DK2019}*{Theorem 4.4.4} that $\rC(K)$ is generated as a C*-algebra by $\rA(K)$ and the bidual $\rC(K)^{**}$ is isomorphic to $\rB(K)$. Moreover, for every point $x \in K_n$, there is a unique *-homomorphism $\delta_x : \rC(K) \to \cM_n$ satisfying $\delta_x|_{\rA(K)} = x$.

\begin{defn} \label{defn:point-evaluation}
Let $K$ be a compact nc convex set. For a point $x \in K_n$, the corresponding {\em point evaluation} $\delta_x : \rC(K) \to \cM_n$ is the unique *-homomorphism satisfying $\delta_x|_{\rA(K)} = x$.
\end{defn}

\subsection{Categorical duality}

The noncommutative Kadison representation theorem from \cite{DK2019}*{Theorem 3.2.3} asserts that every operator system is unitally completely order isomorphic to the space of continuous affine nc functions on a compact nc convex set. More generally, the category of compact nc convex sets is dually equivalent to the category of closed operator systems. To make this precise, we require the following definition from \cite{DK2019}*{Definition 2.5.1}. 

\begin{defn}
Let $K$ and $L$ be compact nc convex sets. A map $\theta : K \to L$ is an {\em affine nc map} if it is graded, respects direct sums and is equivariant with respect to isometries, meaning that
\begin{enumerate}
\item $\theta(K_n) \subseteq L_n$ for all $n$,
\item $\theta(\sum \alpha_i x_i \alpha_i^*) = \sum \alpha_i^* \theta(x_i) \alpha_i$ for every family $\{x_i \in K_{n_i}\}$ and every family of isometries $\{\alpha_i \in \cM_{n_i,n}\}$ satisfying $\sum \alpha_i \alpha_i^* = 1_n$,
\item $\theta(\beta^* y \beta) = \beta^* \theta(y) \beta$ for every $y \in K_m$ and every isometry $\beta \in \cM_{m,n}$. 
\end{enumerate}
An affine nc map $\theta$ is {\em continuous} if the restriction $\theta|_{K_n}$ is continuous for every $n$. We say that $\theta$ is a {\em homeomorphism} and that $K$ and $L$ are {\em affinely homeomorphic} if $\theta$ has a continuous inverse.
\end{defn}

Specifically, let $\mathrm{NCConv}$ denote the category of compact nc convex sets with continuous affine nc maps as morphisms and let $\mathrm{OpSys}$ denote the category of norm closed operator systems with unital complete order homomorphisms as morphisms. There is a contravariant functor $\rA \colon \mathrm{NCConv} \to \mathrm{OpSys}$ sending a compact nc convex set $K$ to the operator system $\rA(K)$ of continuous affine nc functions on $K$. The inverse of the functor $\rA$ is the contravariant functor $\rA^{-1} : \mathrm{OpSys} \to \mathrm{NCConv}$ sends an operator system $S$ to the nc state space of $S$.

The next result is \cite{DK2019}*{Theorem 3.2.5}.

\begin{thm} \label{thm:nc_kadison}
The contravariant functors $\rA : \rA \colon \mathrm{NCConv} \to \mathrm{OpSys}$ and $\rA^{-1} : \mathrm{OpSys} \to \mathrm{NCConv}$ are inverses. In particular, the categories $\mathrm{NCConv}$ and $\mathrm{OpSys}$ are dually equivalent.
\end{thm}

The next result is \cite{DK2019}*{Corollary 3.2.6}.

\begin{cor}
Let $K$ and $L$ be compact nc convex sets. The operator systems $\rA(K)$ and $\rA(L)$ are isomorphic if and only if $K$ and $L$ are affinely homeomorphic.
\end{cor}

\subsection{Convex noncommutative functions}

Let $K$ be a compact nc convex set. An element $f = (f_{ij}) \in \cM_n(\rB(K))$ can be viewed as an nc function $f : K \to \cM_n(\cM)$ defined by $f(x) = (f_{ij}(x))$ for $x \in K$. Note that $f$ is graded, respects direct sums and is unitarily equivariant in an appropriate sense. Note that $f$ is also bounded in an appropriate sense. We say that $f$ is {\em continuous} if $f \in \cM_n(\rC(K))$, and we say that $f$ is {\em self-adjoint} if $f(x) \in \cM_n(\cM_k)_{sa}$ for all $k$ and all $x \in K_k$. For an nc state $\mu : \rC(K) \to \cM_k$, $\mu(f) = (\mu(f_{ij}))$.

The next definition is inspired by the characterization of classical convex functions in terms of their epigraph. It is \cite{DK2019}*{Definition 7.2.1}. 

\begin{defn}
Let $K$ be a compact nc convex set and let $f \in \cM_n(\rB(K))$ be a bounded self-adjoint nc function. The {\em epigraph} of $f$ is the subset $\operatorname{Epi}(f) \subseteq \coprod_m K_m \times \cM_n(\cM_m)$ defined by
\[
\operatorname{Epi}_m(f) = \{(x,\alpha) \in K_m \times \cM_n(\cM_m) : x \in K_m \text{ and } \alpha \geq f(x) \}.
\]
A bounded self-adjoint nc function $f$ is {\em convex} if $\operatorname{Epi}(f)$ is an nc convex set.
\end{defn}

It was observed in \cite{DK2019}*{Remark 7.2.2} that a bounded self-adjoint nc function $f \in \cM_n(\rB(K))$ is convex if and only if
\[
f(\alpha^* x \alpha) \leq (1_n \otimes \alpha^*) f(x) (1_n \otimes \alpha)
\]
for every $l$ and $m$, every $x \in K_m$ and every isometry $\alpha \in \cM_{m,l}$.

\subsection{Representing maps and partial orders}
\begin{defn}
Let $K$ be a compact nc convex set and let $\mu : \rC(K) \to \cM_n$ be an nc state. The {\em barycenter} of $\mu$ is the unique point $x \in K_n$ satisfying $x = \mu|_{\rA(K)}$. The nc state $\mu$ is said to {\em represent} $x$.
\end{defn}

Every point $x \in K_n$ has at least one representing map, namely the point evaluation $\delta_x : \rC(K) \to \cM_n$ from Definition \ref{defn:point-evaluation}. The fact that there may be other nc states representing $x$ is a key part of the theory. The next result combines \cite{DK2019}*{Theorem 5.2.3} and \cite{DK2019}*{Theorem 6.1.6}.

\begin{thm}
Let $K$ be a compact nc convex set. A point $x \in K_n$ is maximal if and only if the corresponding point evaluation $\delta_x : \rC(K) \to \cM_n$ is the unique representing map for $x$. The point $x$ is extreme if and only if it is maximal and $\delta_x$ is irreducible.
\end{thm}

\begin{defn}
Let $\mu \colon \rC(K) \to \cM_m$ be an nc state. We say that a pair $(x,\alpha) \in K_n \times \cM_{n,m}$ is a {\em representation} of $\mu$ if $\alpha$ is an isometry and $\mu = \alpha^* \delta_x \alpha$. The representation $(x,\alpha)$ is {\em minimal} if $\rC(K)(x) \alpha H_m$ is dense in $H_n$.
\end{defn}

For a compact nc convex set $K$, there is an order on the set of nc states on $\rC(K)$ that is analogous to the Choquet order on the set of probability measures on a compact convex set. The next definition is \cite{DK2019}*{Definition 8.2.1}. 

\begin{defn}
Let $K$ be a compact nc convex set and let $\mu,\nu \colon \allowbreak \rC(K) \to \cM_m$ be nc states. We say that $\mu$ is dominated by $\nu$ in the {\em nc Choquet order} and write $\mu \prec_c \nu$ if $\mu(f) \leq \nu(f)$ for every $n$ and every convex nc function $f \in \cM_n(\rC(K))$.
\end{defn}

For each $n$, the nc Choquet order is a partial order on the set of $\cM_n$-valued nc states on $\rC(K)$ \cite{DK2019}*{Proposition 8.2.4}. It is important to note that if $\mu \prec_c \nu$, then $\mu$ and $\nu$ have the same barycenter \cite{DK2019}*{Lemma 8.3.3}.

In this paper, we will be particularly interested in representing maps that are maximal with respect to the nc Choquet order. However, there is another order on the set of nc states on $\rC(K)$, inspired by the dilation order on probability measures introduced in \cite{DK2016}, that is often more convenient to work with. The next definition is \cite{DK2019}*{Definition 8.3.1}. 

\begin{defn}
Let $K$ be a compact nc convex set and let $\mu, \nu \colon \allowbreak \rC(K) \to \cM_l$ be nc states. We say that $\mu$ is dominated by $\nu$ in the {\em dilation order} and write $\mu \prec_d \nu$ if there are representations $(x, \alpha) \in K_m \times \cM_{m,l}$ for $\mu$ and $(y, \beta) \in K_n \times \cM_{n,l}$ for $\nu$ along with an isometry $\gamma \in \cM_{m,n}$, such that $\beta = \gamma \alpha$ and $x = \gamma^* y \gamma$.
\end{defn}

A key result from \cite{DK2019}*{Theorem 8.5.1} is that the nc Choquet order and the dilation order coincide.

\begin{thm}
Let $K$ be a compact nc convex set and let $\mu,\nu : \allowbreak \rC(K) \to \cM_n$ be nc states. Then $\mu \prec_c \nu$ if and only if $\mu \prec_d \nu$. 
\end{thm} 

By \cite{DK2019}*{Corollary 8.3.8} every point $x \in K$ admits a representing map that is maximal in the nc Choquet order (or equivalently, in the dilation order). These states will play an important role in what follows. The next result is \cite{DK2019}*{Theorem 8.3.7}.

\begin{thm}
Let $\mu \colon \rC(K) \to \cM_m$ be an nc state with representation $(x, \alpha) \in K_n \times \cM_{n,m}$. If $\mu$ is maximal in the nc Choquet order and the representation $(x,\alpha)$ is minimal, then $x$ is a maximal point. Conversely, if $x$ is a maximal point, then $\mu$ is maximal.
\end{thm}

\subsection{Minimal C*-algebra}

The minimal C*-algebra associated to the operator system of continuous affine nc functions on a compact nc convex set will play an important role in this paper.

For a compact nc convex set $K$, the minimal C*-algebra $\cmin(\rA(K))$ of the operator system $\rA(K)$ is a C*-algebra that is uniquely determined up to isomorphism by the universal property that there is a unital complete order embedding $\iota : \rA(K) \to \cmin(\rA(K))$ such that $\cmin(\rA(K)) = \ca(\iota(\rA(K)))$ and for any unital C*-algebra $B$ and unital complete order embedding $\phi : \rA(K) \to B$ satisfying $B = \ca(\phi(\rA(K)))$, there is a a surjective homomorphism $\pi : B \to \cmin(\rA(K))$ satisfying $\pi \circ \phi = \iota$.
\[
\begin{tikzcd}
                                        & B=\ca(\phi(\rA(K))) \arrow[d, "\pi", two heads] \\
\rA(K) \arrow[r, "\iota", hook] \arrow[ru, "\phi", hook] & \cmin(\rA(K))                          
\end{tikzcd}
\]

It was observed in \cite{DK2019}*{Section 6.5} that the extreme boundary $\partial K$ can be identified with a (potentially proper) subset of the irreducible representations of $\cmin(\rA(K))$. The next result is \cite{DK2019}*{Theorem 6.5.1}.

\begin{thm}
Let $K$ be a compact nc convex set and let $L$ denote the nc state space of the minimal C*-algebra $\cmin(\rA(K))$, identified with the set of nc states on $\rC(K)$ that factor through $\cmin(\rA(K))$. Then $L$ is the closed nc convex hull of the set of point evaluations $\{\delta_x : x \in \partial K\}$.
\end{thm}

\section{Noncommutative Choquet simplices} \label{sec:nc-simplices}

In this section, we will introduce the definition of an nc Choquet simplex, or briefly an nc simplex, and prove some basic results. At the end of this section we will give some examples. In particular, we will show that every simplex can be naturally identified with an nc simplex.

\begin{defn}
Let $K$ be a compact nc convex set. We will say that $K$ is an {\em nc simplex} if every point in $K$ has a unique representing map on $\rC(K)$ that is maximal in the nc Choquet order.
\end{defn}

\begin{prop} \label{prop:maximal-nc-states-nc-convex}
Let $K$ be a compact nc convex set and let $M$ denote the set of nc states on $\rC(K)$ that are maximal in the nc Choquet order. Then $M$ is an nc convex set.
\end{prop}

\begin{rem}
Note that we do not claim the nc convex set $M$ is closed.
\end{rem}

\begin{proof}
We first show that $M$ is closed under direct sums. For a family $\{\mu_i \in M_{m_i}\}$ of nc states that are maximal in the nc Choquet order and a family $\{ \alpha_i \in \cM_{m,m_i} \}$ of isometries satisfying $\sum \alpha_i \alpha_i^* = 1_m$, let $\mu : \rC(K) \to \cM_m$ denote the nc state defined by $\mu = \sum \alpha_i \mu_i \alpha_i^*$. Let $\{(x_i, \xi_i) \in K_{n_i} \times \cM_{n_i, m_i} \}$ be a family such that each $(x_i,\xi_i)$ is a minimal representation of $\mu_i$. Let $\beta_i \in \cM_{n,n_i}$ be isometries satisfying $\sum \beta_i \beta_i^* = 1_n$, let $x = \sum \beta_i x_i \beta_i^* \in K_n$ and let $\xi = \sum \beta_i \xi_i \alpha_i^* \in \cM_{n,m}$. Then $(x,\xi)$ is a representation for $\mu$. By \cite{DK2019}*{Theorem 8.3.7}, the maximality of each map $\mu_i$ along with the minimality of the corresponding representation $(x_i,\xi_i)$ implies that each point $x_i$ is maximal. Since direct sums of maximal points are maximal, $x$ is a maximal point. Applying \cite{DK2019}*{Theorem 8.3.7} again, we conclude that $\mu$ is maximal in the nc Choquet order.

Next, we show that $M$ is closed under compressions. For a maximal nc state $\nu \in M_m$ that is maximal in the nc Choquet order and an isometry $\beta \in \cM_{m,l}$, let $\lambda = \beta^* \nu \beta$. Let $(y,\eta) \in K_n \times \cM_{n,m}$ be a minimal representation for $\nu$. Then $(y, \eta \beta)$ is a (not necessarily minimal) representation for $\nu$. By \cite{DK2019}*{Theorem 8.3.7}, the maximality of $\nu$ and the minimality of the corresponding representation $(y,\eta)$ implies that the point $y$ is maximal. Applying \cite{DK2019}*{Theorem 8.3.7} again, we conclude that $\lambda$ is maximal in the nc Choquet order.
\end{proof}

The next result implies that in order to verify that a compact nc convex set $K$ is an nc simplex, it suffices to look at the points in $K_1$.

\begin{prop}
Let $K$ be a compact nc convex set. Then $K$ is an nc simplex if and only if every point in $K_1$ has a unique representing map on $\rC(K)$ that is maximal in the nc Choquet order.
\end{prop}

\begin{proof}
If $K$ is an nc simplex, then it follows immediately that every point in $K_1$ has a unique nc Choquet maximal representing map on $\rC(K)$.

Conversely, suppose that every point in $K_1$ has a unique nc Choquet maximal representing map on $\rC(K)$ and let $y \in K_n$ be an arbitrary point. Let $\nu,\nu' : \rC(K) \to \cM_n$ be nc Choquet maximal representing maps for $y$. Let $\alpha \in \cM_{n,1}$ be an isometry and let $x = \alpha^* y \alpha$. Then $x \in K_1$, so by assumption it has a unique representing map that is maximal in the nc Choquet order.

Let $\mu = \alpha^* \nu \alpha$ and $\mu' = \alpha^* \nu' \alpha$. Then $\mu$ and $\mu'$ are representing maps for $x$. Furthermore, it follows from Proposition \ref{prop:maximal-nc-states-nc-convex} that they are maximal in the nc Choquet order since $\nu$ and $\nu'$ are. Hence $\mu = \mu'$. Since $\alpha$ was arbitrary, it follows that $\nu = \nu'$. Hence $y$ has a unique nc Choquet maximal representing maps. Since $y$ was arbitrary, $K$ is an nc simplex.
\end{proof}

\begin{prop} \label{prop:map-to-max-rep-measure-maximal}
Let $K$ be an nc simplex and let $M$ denote the set of maximal nc states on $\rC(K)$. For $x \in K_n$, let $\mu_x \in M_n$ denote the unique nc Choquet maximal representing maps for $x$. Then the map $\theta : K \to M$ defined by $\theta(x) = \mu_x$ is an affine nc map.
\end{prop}

\begin{rem}
Let $L$ denote the nc state space of $\rC(K)$. The map $\theta$ is a section of the barycenter map which maps $\mu \in L$ to the restriction $\mu|_{\rA(K)} \in K$. However, $\theta$ is not necessarily continuous.
\end{rem}

\begin{proof}
It is clear that $\theta(K_n) \subseteq M_n$ for each $n$.

For a family of points $\{x_i \in K_{n_i} \}$ and a family of isometries $\{ \alpha_i \in \cM_{n_i,n} \}$ satisfying $\sum \alpha_i \alpha_i^* = 1_n$, the map $\sum \alpha_i \mu_{x_i} \alpha_i^*$ has barycenter $\sum \alpha_i x_i \alpha_i^*$. Furthermore, since each $\mu_i$ is maximal, it follows from Proposition \ref{prop:maximal-nc-states-nc-convex} that $\sum \alpha_i \mu_{x_i} \alpha_i^*$ is maximal in the nc Choquet order. By uniqueness, $\theta(\sum \alpha_i x_i \alpha_i^*) = \sum \alpha_i \mu_{x_i} \alpha_i^* = \sum \alpha_i \theta(x_i) \alpha_i^*$.

Similarly, for $y \in K_n$ and an isometry $\beta \in \cM_{m,n}$, the map $\beta^* \mu_y \beta$ has barycenter $\beta^* y \beta$, and it follows from Proposition \ref{prop:maximal-nc-states-nc-convex} that $\beta^* \mu_y \beta$ is maximal. Hence by uniqueness, $\theta(\beta^* x \beta) = \beta^* \mu_x \beta = \beta^* \theta(x) \beta$.
\end{proof}

The existence of the map provided by the next result is the key technical property satisfied by nc simplices.

\begin{thm} \label{thm:characterization-nc-simplex-expectation}
Let $K$ be a compact nc convex set. Then $K$ is an nc simplex if and only if there is a unital completely positive map $\phi : \rC(K) \to \rA(K)^{**}$ such that $\phi|_{\rA(K)} = \id_{\rA(K)}$, meaning that the following diagram commutes:
\[
\begin{tikzcd}
\rC(K) \arrow[rr, "\phi"] & & \rA(K)^{**} \\
& \arrow[lu, "", hook] \rA(K) \arrow[ru, "", hook] &  
\end{tikzcd}
\]
If the map $\phi$ exists, then it is unique and factors through $\cmin(\rA(K))$. Moreover, $\phi$ extends to a normal conditional expectation $\phi^{**} : \rC(K)^{**} \to \rA(K)^{**}$ with range $\rA(K)^{**}$
\end{thm}

\begin{proof}
Suppose $K$ is an nc Choquet simplex. For $n \in \bN$ and $f \in \cM_n(\rC(K))$, define a function $\phi(f) : K \to \cM_n(\cM)$ by $\phi(f)(x) = \mu_x(f)$ for $x \in K_k$, where $\mu_x : \rC(K) \to \cM_k$ denotes the unique nc Choquet maximal representing map for $x$. Then it follows from Proposition \ref{prop:map-to-max-rep-measure-maximal} that $\phi(f)$ is a bounded affine nc function. For $f \geq 0$, it is clear that $\phi(f) \geq 0$. Moreover, it is also clear that $\phi(a) = a$ for $a \in \rA(K)$. Hence $\phi : \rC(K) \to \rA(K)^{**}$ defines a unital completely positive map satisfying $\phi|_{\rA(K)} = \id_{\rA(K)}$.

For $x \in K_k$, let $(y,\alpha) \in K_m \times \cM_{m,k}$ be a minimal representation for $\mu_x$, so that $\mu_x = \alpha^* \delta_y \alpha$. Since $\mu_x$ is maximal, \cite{DK2019}*{Theorem 8.3.7} implies that $y$ is a maximal point. By \cite{DK2019}*{Proposition 5.2.4}, $\delta_y$ factors through $\cmin(\rA(K))$. Hence $\mu_x$ factors through $\cmin(\rA(K))$. It follows that $\phi$ factors through $\cmin(\rA(K))$.

Conversely, let $\phi : \rC(K) \to \rA(K)^{**}$ be a unital completely positive map such that $\phi|_{\rA(K)} = \id_{\rA(K)}$. Then for $y \in K_m$, the map $y^{**} \circ \phi$ has barycenter $y$. If $y$ is a maximal point, then it has a unique representing map $\delta_y : \rC(K) \to \cM_m$, so in this case $y^{**} \circ \phi = \delta_y$. For arbitrary $x \in K_k$, let $\mu : \rC(K) \to \cM_k$ be a representing map for $x$ that is maximal in the nc Choquet order and let $(y,\alpha) \in K_m \times \cM_{m,k}$ be a minimal representation for $\mu$. Then \cite{DK2019}*{Theorem 8.3.7} implies that $y$ is a maximal point, so from above, 
\[
\mu = \alpha^* \delta_y \alpha = \alpha^* (y^{**} \circ \phi) \alpha = (\alpha^* y \alpha) \circ \phi = x^{**} \circ \phi.
\]
In particular, $\mu$ is unique. Hence $K$ is an nc Choquet simplex. Note that this argument also implies the uniqueness of the map $\phi$.

Finally, since $\rA(K)^{**}$ is a dual operator system, we obtain a unique extension of $\varphi$ to a normal unital completely positive map $\varphi^{**} : \rC(K)^{**} \to \rA(K)^{**}$. It follows from above that $\varphi^{**}|_{\rA(K)^{**}} = \id_{\rA(K)^{**}}$. Hence $\varphi^{**}$ is a conditional expectation.
\end{proof}

The next result describes the unique nc Choquet maximal representing maps for points in an nc simplex. It follows immediately from the proof of Theorem \ref{thm:characterization-nc-simplex-expectation}.

\begin{cor} \label{cor:formula-affinization-map}
Let $K$ be an nc simplex and let $\phi : \rC(K) \to \rA(K)^{**}$ denote the unital completely positive map from Theorem \ref{thm:characterization-nc-simplex-expectation}. For $x \in K_n$, let $\mu_x : \rC(K) \to \cM_n$ denote the unique nc Choquet maximal representing map for $x$. Then $\mu_x = x^{**} \circ \phi$.
\end{cor}

\section{Basic examples} \label{sec:basic-examples}

In this section we will give the first examples of nc simplices by applying the characterization from Theorem \ref{thm:characterization-nc-simplex-c-star-system}. In particular, we will show that every simplex can naturally be identified with an nc simplex.

An operator system $S$ is said to have the {\em weak expectation property} if there is an injective operator system $T$ such that $S \subseteq T \subseteq S^{**}$ . This definition, introduced by Kavruk \cite{Kav2012}*{Section 3.2} in a slightly different form, is a generalization of Lance's weak expectation property for C*-algebras (see \cite{BO2008}*{Chapter 13}).

\begin{thm} \label{thm:weak-exp-simplex}
The nc state space of an operator system with the weak expectation property is an nc simplex.
\end{thm}

\begin{proof}
Let $S$ be an operator system with the weak expectation property and let $K$ denote the nc state space of $S$, so that $S$ is unitally completely order isomorphic to $\rA(K)$. Let $T$ be an injective operator system with $\rA(K) \subseteq T \subseteq \rA(K)^{**}$. By the injectivity of $T$, the identity map on $\rA(K)$ extends to a unital completely positive map $\psi : \rC(K) \to \rA(K)^{**}$. The result now follows from Theorem \ref{thm:characterization-nc-simplex-c-star-system}.
\end{proof}

An operator system $S$ is said to be {\em nuclear} if $S \tmin T = S \tmax T$ for every operator system $T$. Here, $S \tmin T$ and $S \tmax T$ denote the minimal and maximal tensor products respectively of $S$ and $T$ (see Section \ref{sec:tensor-products}). By \cite{HP2011}*{Theorem 3.5}, $S$ is nuclear if and only if the bidual $S^{**}$ is unitally completely order isomorphic to an injective C*-algebra. In particular, if $S$ is nuclear, then it has the weak expectation property. The next result follows immediately from this observation and Theorem \ref{thm:weak-exp-simplex}.

\begin{cor} \label{cor:nuclear-simplex}
The nc state space of a nuclear operator system is an nc simplex.
\end{cor}

It follows from Theorem \ref{thm:weak-exp-simplex} that the nc state space of a unital C*-algebra with the weak expectation property is an nc simplex. However, we will establish a much stronger result in Section \ref{sec:nc-bauer-simplices}. In the next example, we apply Corollary \ref{cor:nuclear-simplex} to give an example of an nc simplex that is not affinely homeomorphic to the nc state space of a C*-algebra.

\begin{example}
Han and Paulsen \cite{HP2011}*{Theorem 4.2} constructed the following example of a nuclear operator system that is not completely order isomorphic to a C*-algebra. Let $S \subseteq \B(\ell^2(\bN))$ denote the norm closure of
\[
\spn\{1,e_{ij} : (i,j) \in \bN^2 \setminus \{(1,1)\} \},
\]
where $\{e_{ij}\}_{i,j \in \bN}$ is the family of standard matrix units in $\B(\ell^2(\bN))$. Then $S$ is a closed operator subsystem of $\B(\ell^2(\bN))$ that not unitally completely order isomorphic to a C*-algebra. However, by \cite{HP2011}*{Theorem 4.3}, the bidual $S^{**}$ is unitally completely order isomorphic to the injective C*-algebra $\B(\ell^2(\bN))$. Hence by \cite{HP2011}*{Theorem 3.5}, $S$ is nuclear.

Letting $K$ denote the nc state space of $S$ so that $S$ is unitally completely order isomorphic to $\rA(K)$, it follows from Corollary \ref{cor:nuclear-simplex} that $K$ is an nc simplex. 

\end{example}

For a compact convex set $C$ let $\max(C)$ denote the largest compact nc convex set with $\max(C)_1 = C$. Equivalently, $\max(C)$ is the nc state space of the function system $\rA(C)$ of continuous affine functions on $C$ when it is equipped with its minimal operator system structure (cf. \cite{PTT2010}), i.e. when it is viewed as an operator subsystem of the C*-algebra $\rC(C)$ of continuous functions on $C$.

\begin{lem} \label{lem:criteria-max-nc-convex-set}
Let $K$ be a compact nc convex set. The following are equivalent:
\begin{enumerate}
\item $K = \max(K_1)$
\item $\partial K = \partial K_1$
\item $\cmin(\rA(K)) = \rC(\ol{\partial K_1})$
\end{enumerate}
\end{lem}

\begin{proof}
By the duality between the category of compact nc convex sets and operator systems, $K = \max(K_1)$ if and only if $\rA(K) = \rA(\max(K_1))$. Since $\rA(\max(K_1))$ is the operator system obtained by equipping the function system $\rA(K_1)$ with its minimal operator system structure, i.e. with the operator system structure obtained from the inclusion $\rA(K_1) \subseteq \rC(K_1)$, it follows that $\rA(K) = \rA(\max(K_1))$ if and only if the canonical map from $\rA(K)$ into $\rC(\ol{\partial K_1})$ is a unital complete order embedding, which is equivalent to the statement that $\cmin(\rA(K))$ is a quotient of $\rC(\ol{\partial K_1})$. 

By \cite{DK2015}*{Theorem 2.4}, every point in $\partial K_1$ can be dilated to an extreme point of $K$. It follows from \cite{DK2015}*{Theorem 3.4}, that $\cmin(\rA(K))$ is a quotient of $\rC(\ol{\partial K_1})$ if and only if $\cmin(\rA(K)) = \rC(\ol{\partial K_1})$, and that this is equivalent to $\partial K = \partial K_1$.
\end{proof}

The next result shows that every simplex can be naturally identified with an nc simplex.

\begin{thm} \label{thm:nc-criterion-classical-simplex}
Let $K$ be a compact nc convex set. Then $K$ is an nc simplex with $\partial K = \partial K_1$ if and only if $K_1$ is a simplex and $K = \max(K_1)$.
\end{thm}

\begin{proof}
Suppose that $K$ is an nc simplex with $\partial K = \partial K_1$. Then by Lemma \ref{lem:criteria-max-nc-convex-set}, $K = \max(K_1)$ and $\cmin(\rA(K)) = \rC(\ol{\partial K_1})$. In particular, $\cmin(\rA(K))$ is a commutative C*-algebra. Let $\phi : \rC(K) \to \rA(K)^{**}$ denote the map from Theorem \ref{thm:characterization-nc-simplex-expectation}. Since $\phi$ factors through $\cmin(\rA(K))$, the normal conditional expectation $\phi^{**}~:~\rC(K)^{**} \to \rA(K)^{**}$ gives rise to a conditional expectation $\rC(\ol{\partial K_1})^{**} \to \rA(K_1)^{**}$. Since $\rC(\ol{\partial K_1})^{**}$ is a commutative von Neumann algebra, it is injective. It follows that $\rA(K_1)^{**}$ is injective, and hence by \cite{HP2011}*{Theorem 3.5}, $\rA(K_1)$ is nuclear. Therefore, by \cite{NP1969}*{Theorem 1.4}, $K_1$ is a simplex. 

Conversely, suppose that $K_1$ is a simplex and $K = \max(K_1)$. Since $K_1$ is a simplex, \cite{Eff1972}*{Theorem 7.4} implies that the bidual $\rA(K_1)^{**}$ is an injective function system, and hence is unitally order isomorphic to a commutative von Neumann algebra. Therefore, $\rA(K)^{**}$ is unitally completely order isomorphic to a commutative von Neumann algebra, and in particular is injective as an operator system. By \cite{HP2011}*{Theorem 3.5}, $\rA(K)$ is nuclear, so it follows from Corollary \ref{cor:nuclear-simplex} that $K$ is an nc simplex.
\end{proof}

\section{C*-systems} \label{sec:c-star-systems}

Recall that the category of compact nc convex sets with morphisms consisting of continuous affine nc maps is dual to the category of closed operator systems with morphisms consisting of unital complete order homomorphisms (see \cite{DK2019}*{Section 3}). This duality is implemented by the functor mapping a compact nc convex set $K$ to the operator system $\rA(K)$ of continuous affine nc functions on $K$.

In this section we will characterize nc simplices in terms of their corresponding operator systems of continuous affine nc functions. Specifically, we will show that a compact nc convex set is an nc simplex if and only if it is dual to a special kind of operator system introduced by Kirchberg and Wasserman \cite{KW1998} called a C*-system.

\begin{defn}[C*-system]
An operator system $S$ is said to be a {\em C*-system} if its bidual $S^{**}$ is unitally completely order isomorphic to a von Neumann algebra.
\end{defn}

\begin{thm} \label{thm:characterization-nc-simplex-c-star-system}
Let $K$ be a compact nc convex set. Then $K$ is an nc simplex if and only if the corresponding operator system $\rA(K)$ of continuous affine nc functions on $K$ is a C*-system.
\end{thm}

\begin{proof}
Suppose $K$ is an nc simplex and let $\phi^{**} : \rC(K)^{**} \to \rA(K)^{**}$ denote the normal conditional expectation from Theorem \ref{thm:characterization-nc-simplex-expectation}. The range $\rA(K)^{**}$ of $\phi^{**}$ is a C*-algebra with respect to the corresponding Choi-Effros product \cite{CE1977}*{Theorem 3.1}. Since $\rA(K)^{**}$ is weak*-closed, it follows from Sakai's characterization of von Neumann algebras as C*-algebras with preduals that this C*-algebra is a von Neumann algebra.

Conversely, suppose that $\rA(K)$ is a C*-system. Then there is a von Neumann algebra $M$ and a unital complete order isomorphism $\psi : \rA(K)^{**} \to M$. Let $B = \ca(\psi(\rA(K)))$. By the universal property of $\rC(K)$, there is a surjective homomorphism $\pi : \rC(K) \to B$ such that $\pi|_{\rA(K)} = \psi|_{\rA(K)}$. Let $\phi = \psi^{-1} \circ \pi$. Then $\phi : \rC(K) \to \rA(K)^{**}$ is a unital completely positive map satisfying $\phi|_{\rA(K)} = \id_{\rA(K)}$. Hence by Theorem \ref{thm:characterization-nc-simplex-expectation}, $K$ is an nc simplex. 
\end{proof}

The next example is due to Kirchberg and Wassermann \cite{KW1998}.

\begin{example}[Kirchberg-Wassermann]
Let $A$ be a unital C*-algebra and let $L$ be a closed left ideal in $A$. Then $L + L^*$ is a closed subspace of $A$ and there is a natural operator system structure on the quotient $A/(L + L^*)$ obtained from the canonical isometric embedding of $A / (L + L^*)$ into the operator system $(1-p)A^{**}(1-p)$, where $p \in A^{**}$ denotes the support projection of $L$. Furthermore, $(A/(L+L^*))^{**}$ is unitally completely order isomorphic to the C*-algebra $(1-p)A^{**}(1-p)$. Hence $A/(L+L^*)$ is a C*-system. Kirchberg and Wassermann showed that every separable C*-system is of this form \cite{KW1998}*{Proposition 5}, and that every non-separable C*-system is an inductive limit of separable C*-systems.
\end{example}

\begin{cor} \label{cor:phi-homomorphism}
Let $K$ be an nc simplex and let $\phi : \rC(K) \to \rA(K)^{**}$ denote the unital completely positive map from Theorem \ref{thm:characterization-nc-simplex-expectation}. Then $\phi$ is a homomorphism with respect to the von Neumann algebra structure on $\rA(K)^{**}$ from Theorem \ref{thm:characterization-nc-simplex-c-star-system}.
\end{cor}

\begin{proof}
If $K$ is an nc simplex, then arguing as in the proof of Theorem \ref{thm:characterization-nc-simplex-c-star-system}, there is a von Neumann algebra $M$ and a unital complete order isomorphism $\psi : \rA(K)^{**} \to M$. Let $B = \ca(\psi(\rA(K)))$. By the universal property of $\rC(K)$, there is a surjective homomorphism $\pi : \rC(K) \to B$ such that $\pi|_{\rA(K)} = \psi|_{\rA(K)}$. The unital completely positive map $\phi' : \rC(K) \to \rA(K)^{**}$ defined by $\phi' = \psi^{-1} \circ \pi$ satisfies $\phi'|_{\rA(K)} = \id_{\rA(K)}$. Hence by Theorem \ref{thm:characterization-nc-simplex-expectation}, $\phi' = \phi$.
\end{proof}

It follows immediately from Theorem \ref{thm:characterization-nc-simplex-c-star-system} that the nc state space of a unital C*-algebra is an nc simplex. In Section \ref{sec:nc-bauer-simplices}, we will completely characterize nc simplices that arise in this way.

\begin{cor} \label{cor:nc-state-space-c-star-alg-prelim}
The nc state space of a unital C*-algebra is an nc simplex. 
\end{cor}

For $n \geq 1$, a compact convex set $C$ is an $n$-simplex if and only if the function system $\rA(C)$ is unitally order isomorphic to $\ell^\infty_n$, which is a finite-dimensional C*-algebra. The next result is a noncommutative generalization of this fact.

\begin{cor} \label{cor:fin-dim-choquet-simplex-c-star-alg}
Let $K$ be compact nc convex set such that $\rA(K)$ is finite dimensional. Then $K$ is an nc simplex if and only if $\rA(K)$ is unitally completely order isomorphic to a unital C*-algebra.
\end{cor}

\begin{proof}
The result follows immediately from Corollary \ref{cor:nc-state-space-c-star-alg-prelim} and the fact that $\rA(K)^{**} = \rA(K)$ if $\rA(K)$ is finite dimensional.
\end{proof}

\section{Tensor products} \label{sec:tensor-products}

Namioka and Phelps \cite{NP1969} characterized simplices in terms of tensor products. Specifically, they proved that a compact convex set $C$ is a simplex if and only if the corresponding function system $\rA(C)$ of continuous affine functions on $C$ is nuclear, meaning that $\rA(C) \tmin F = \rA(C) \tmax F$ for every function system $F$. Here, $\rA(C) \tmin F$ and $ \rA(C) \tmax F$ denote the minimal and maximal tensor products respectively of the function systems $\rA(C)$ and $F$ in the category of function systems (see e.g. \cites{BW1970,NP1969}).

In this section, we will establish an analogous characterization of nc simplices in terms of tensor products. We will utilize the theory of tensor products of operator systems developed by Kavruk, Paulsen, Todorov and Tomforde \cite{KPTT2011}.

For operator systems $S$ and $T$, an operator system tensor product is an operator system structure on the algebraic tensor product $S \otimes T$, i.e. a sequence of positive cones $\{C_n \subseteq \cM_n(S \otimes T) \}_{n \in \bN}$. If $R$ is an operator system containing the algebraic tensor product $S \otimes T$, then we obtain an operator system structure on $S \otimes T$ by taking $C_n = \cM_n(R)^+$ for $n \in \bN$.

\begin{defn}
Let $S$ and $T$ be operator systems with nc state spaces $K$ and $L$ respectively.
\begin{enumerate}
\item The {\em minimal tensor product} $S \tmin T$ is the operator system with positive cones $\{C_n^{\mathrm{min}} \subseteq \cM_n(S \otimes T) \}_{n \in \bN}$ defined by
\begin{align*}
C_n^{\mathrm{min}} = \{ &a \in \cM_n(S \otimes T) : a(x \otimes y) \geq 0 \text{ for all } x \in K, \ y \in L \}.
\end{align*}

\item The {\em commuting tensor product} $S \tcomm T$ is the operator system with positive cones $\{C_n^{\mathrm{c}} \subseteq \cM_n(S \otimes T) \}_{n \in \bN}$ defined by
\begin{align*}
C_n^{\mathrm{c}} = \{ &a \in \cM_n(S \otimes T) : a(x \otimes y) \geq 0 \text{ for all } x \in K, \ y \in L \\
& \quad \text{ such that } S(x) \text{ and } T(y) \text{ commute} \}.
\end{align*}

\item The {\em maximal tensor product} $S \tmin T$ is the operator system with positive cones $\{C_n^{\mathrm{max}} \subseteq \cM_n(S \otimes T) \}_{n \in \bN}$ defined by
\begin{align*}
C_n^{\mathrm{max}} = \{ &\alpha(s \otimes t)\alpha^* : s \in \cM_l (S)^+,\ t \in \cM_m(T)^+, \\
& \quad \alpha \in \cM_{n,km},\ l,m \in \bN \}.
\end{align*}

\end{enumerate}
\end{defn}

The following tensor product characterization of C*-systems was established by Kavruk \cite{Kav2018}*{Theorem 2.3}.

\begin{thm}[Kavruk] \label{thm:kavruk}
An operator system $S$ is a C*-system if and only if it is (c,max)-nuclear, meaning that $S \tcomm T = S \tmax T$ for every operator system $T$. Equivalently, if and only if $\cmax(S) \tmax \cmax(T) = \cmax(S \tmax T)$ for every operator system $T$.
\end{thm}

The next corollary follows immediately from Theorem \ref{thm:characterization-nc-simplex-c-star-system} and Theorem \ref{thm:kavruk}.

\begin{cor} \label{cor:tensor-product-char-nc-simplices}
Let $K$ be a compact nc convex set. Then $K$ is an nc simplex if and only if the operator system $\rA(K)$ is $(\operatorname{c},\operatorname{max})$-nuclear, meaning that $\rA(K) \tcomm \rA(L) = \rA(K) \tmax \rA(L)$ for every  compact nc convex set $L$. Equivalently, if and only if $\rC(K) \tmax \rC(L) = \cmax(\rA(K) \tmax \rA(L))$ for every compact nc convex set $L$.
\end{cor}

The next result can be seen as a noncommutative analogue of \cite{NP1969}*{Theorem 1.2}.

\begin{thm}
Let $K$ be an nc Choquet simplex and let $L$ be a compact nc convex set. Let $M$ denote the nc state space of the maximal tensor product $\rA(K) \tmax \rA(L)$. For an extreme point $z \in (\partial M)_n$, there are maximal points $x \in K_n$ and $y \in L_n$ such that $z = x \otimes y$ and the corresponding representations $\delta_x : \rC(K) \to \cM_n$ and $\delta_y : \rC(L) \to \cM_n$ are commuting factor representations. If either of the C*-algebras $\cmin(\rA(K))$ or $\cmin(\rA(L))$ are of type I, then $x \in \partial K$ and $y \in \partial L$. 
\end{thm}

\begin{proof}
Let $z \in (\partial M)_n$ be an extreme point. By Corollary \ref{cor:tensor-product-char-nc-simplices}, $\rC(K) \tmax \rC(L) = \rC(\rA(K) \tmax \rA(L))$. Hence there are points $x \in K_n$ and $y \in L_n$ such that the corresponding representations $\delta_x : \rC(K) \to \cM_n$ and $\delta_y : \rC(L) \to \cM_n$ have commuting ranges and $\delta_z = \delta_x \times \delta_y$ (see e.g. \cite{BO2008}*{Theorem 3.2.6}).

Let $\mu_x : \rC(K) \to \cM_n$ denote the unique nc Choquet maximal representing map for $x$ and let $\nu : \rC(K) \to \cM_n$ be an nc Choquet maximal representing map for $y$. Then there is a unital completely positive map $\lambda = \mu_x \otimes \nu : \rC(K) \maxt \rC(L) \to \cM_n$ such that $\lambda|_{\rC(K)} = \mu_x$ and $\lambda|_{\rC(L)} = \nu$ (see e.g. \cite{BO2008}*{Theorem 3.5.3}). In particular, since $\mu_x$ has barycenter $x$ and $\nu$ has barycenter $y$, $\lambda$ has barycenter $z$, i.e. $\lambda$ is a representing map for $z$. Hence by \cite{DK2019}*{Proposition 8.3.6}, $\delta_z \prec_c \lambda$.

Since $z$ is an extreme point, \cite{DK2019}*{Proposition 5.2.3} implies that $\delta_z$ is the unique representing map for $z$. Hence $\lambda = \delta_z$, implying $\mu_x = \delta_x$ and $\nu = \delta_y$. Hence $\delta_z = \delta_x \times \delta_y$ and $z = x \times y$. Furthermore, applying  \cite{DK2019}*{Proposition 5.2.3} again implies that $x$ and $y$ are maximal points.

Since $z$ is extreme, $\delta_z$ is irreducible. It is a standard fact that in this case $\delta_x$ and $\delta_y$ are factor representations. Since $x$ and $y$ are maximal points, \cite{DK2019}*{Proposition 5.2.4} implies that the representations $\delta_x$ and $\delta_y$ factor through $\cmin(\rA(K))$ and $\cmin(\rA(L))$ respectively.

If either of $\cmin(\rA(K))$ or $\cmin(\rA(L))$ are of type I, then $\delta_x$ and $\delta_y$ are both irreducible, so in this case \cite{DK2019}*{Theorem 6.1.6} implies $x \in \partial K$ and $y \in \partial L$.
\end{proof}

\begin{example}
The points $x \in K_n$ and $y \in L_n$ in the factorization $z = x \times y$ are not necessarily extreme points. For example, let $\bF_2 = \langle u,v \rangle$ denote the free group on two generators. Let $K$ denote the nc state space of $\ca(\bF_2)$, so that $\ca(\bF_2)$ is unitally completely order isomorphic to $\rA(K)$.

Let $\lambda : \ca(\bF_2) \to \B(\ell^2(\bF_2))$ and $\rho : \ca(\bF_2) \to \B(\ell^2(\bF_2))$ denote the extensions to $\ca(\bF_2)$ of the left and right regular representations of $\bF_2$ respectively. By the universal property of the maximal tensor product (see e.g. \cite{BO2008}*{Proposition 3.3.7}), we obtain the biregular representation $\lambda \times \rho : \ca(\bF_2) \maxt \ca(\bF_2) \to \B(\ell^2(\bF_2))$. Since $\bF_2$ has the infinite conjugacy class property, meaning that every non-trivial conjugacy class is infinite, $\lambda$ and $\rho$ are factor representations and $\lambda \times \rho$ is irreducible. 

Let $x,y \in K$ denote the points corresponding to $\lambda$ and $\rho$ respectively. Let $L$ be the nc state space of $\ca(\bF_2) \maxt \ca(\bF_2)$ and let $z \in L$ denote the point corresponding to $\lambda \times \rho$. Then $z = x \times y$. Moreover, since $\lambda \times \rho$ is irreducible, \cite{DK2019}*{} implies that $z \in \partial L$. However, the representations $\lambda$ and $\rho$ are not irreducible, so $x \notin \partial K$ and $y \notin \partial K$. 
\end{example}

\section{Universal measurability} \label{sec:universally-measurable}

In this section, we will prove a key technical result about the map from Theorem \ref{thm:characterization-nc-simplex-expectation}. Specifically, we will show that if $K$ is an nc simplex and $\phi : \rC(K) \to \rA(K)^{**}$ is the unital completely positive map from Theorem \ref{thm:characterization-nc-simplex-expectation}, then $\varphi$ maps convex functions to universally measurable elements in $\cmin(\rA(K))^{**}$. We will require this result in our characterization of compact nc convex sets that are affinely homeomorphic to the nc state space of a C*-algebra in Section \ref{sec:nc-bauer-simplices}.

\begin{lem} \label{lem:breve-function}
Let $K$ be an nc Choquet simplex. For $x \in K_1$, let $\mu_x : \rC(K) \to \bC$ denote the unique nc Choquet maximal representing state for $x$. Let $f \in \rC(K)$ be a self-adjoint continuous nc function, such that $-f$ is convex. Then the function $\breve{f} : K_1 \to \bC$ defined by
\[
\breve{f}(x) = \mu_x(f), \qfor x \in K_1
\]
is lower semicontinuous and affine.
\end{lem}

\begin{proof}
The fact that $\breve{f}$ is affine follows from Proposition \ref{prop:map-to-max-rep-measure-maximal}. Therefore, it remains to show that $\breve{f}$ is lower semicontinuous, i.e. that for a point $x \in K_1$ and a net $\{x_i\}$ in $K_1$ such that $\lim x_i = x$, $\liminf \breve{f}(x_i) \geq \breve{f}(x)$.

Since $-f$ is convex and $\mu_x$ is the unique representing state for $x$ that is maximal in the nc Choquet order, it follows from the definition of the nc Choquet order that for $x \in K_1$,
\[
\breve{f}(x) = \mu_x(f) = \inf_{\nu} \nu(f),
\]
where the infimum is taken over all states $\nu$ on $\rC(K)$ with barycenter $x$. 

Fix $x \in K_1$ and let $\{x_i\}$ be a net in $K_1$ such that $\lim x_i = x$. For $\epsilon > 0$, \cite{DK2019}*{Theorem 8.4.1} implies that for each $i$ there is a state $\nu_i$ on $\rC(K)$ with barycenter $x_i$ such that $\nu_i(f) < \breve{f}(x_i) + \epsilon$. If $\{\nu_j\}$ is any subnet and $\nu$ is a state on $\rC(K)$ with $\lim \nu_j = \nu$, then the continuity of the barycenter map implies that $\nu$ has barycenter $x$. Hence from above,
\[
\breve{f}(x) \leq \nu(f) = \lim \nu_j(f) \leq \liminf \breve{f}(x_j) + \epsilon. 
\]
Since $\epsilon$ was arbitrary, we conclude that $\breve{f}$ is lower semicontinuous.
\end{proof}

Let $A$ be a C*-algebra. A self-adjoint element $a \in A^{**}$ is said to be {\em universally measurable} in $A^{**}$ if for every state $\phi$ on $A$ and every $\epsilon > 0$, there are self-adjoint elements $b,c \in A^{**}$ such that $b \leq a \leq c$ and $\phi(c - b) < \epsilon$ (cf. \cite{Ped1979}*{4.3.11}). In Section \ref{sec:nc-bauer-simplices}, we will require the result that the atomic representation of $A$ is faithful on the set of universally measurable elements in $A^{**}$ (see e.g. \cite{Ped1979}*{4.3.15}).

\begin{lem} \label{lem:universally-measurable}
Let $K$ be an nc Choquet simplex. Let $\phi : \rC(K) \to \rA(K)^{**}$ denote the unital completely positive map from Theorem \ref{thm:characterization-nc-simplex-expectation}. Let $\iota : \rA(K) \to \cmin(\rA(K))$ denote the canonical unital complete order embedding and let $\iota^{**} : \rA(K)^{**} \to \cmin(\rA(K))^{**}$ denote the unique normal extension of $\iota$. Let $S \subset C(K)$ be the norm closure of the subspace spanned by the convex self-adjoint elements. Then $(\iota^{**} \circ \varphi(S)$ is contained in the universally measurable elements of $\cmin(\rA(K))^{**}$.
\end{lem}

\begin{proof}
Let $f \in \rC(K)$ be a self-adjoint continuous nc function with the property that $-f$ is convex. Note that by \cite{DK2019}*{Corollary 4.4.4} $\rA(K)^{**}$ can be identified with bounded affine nc functions on $K$. In particular, $\varphi(f)(x) = x^{**}(\varphi(f))$. By Corollary \ref{cor:formula-affinization-map}, $\phi(f)(x) = \mu_x(f)$ for $x \in K_1$. Hence $\phi(f)|_{K_1} = \breve{f}$, where $\breve{f} : K_1 \to \bC$ is the lower semicontinuous affine function from Lemma \ref{lem:breve-function}. 

Consider the element $b = \varphi(f) \in \rA(K)^{**} \subset \cmin(\rA(K))^{**}$. Let $L_1$ denote the state space of $\cmin(\rA(K))$. By \cite{Ped1979}*{3.11.8} $b$ belongs to the monotone closure of $\cmin(\rA(K))$ in $\cmin(\rA(K))^{**}$ if and only if the function on $L_1$ defined by $\hat{b}(\mu) = \mu^{**}(b)$ is lower semi-continuous. Let $\eta \colon L_1 \to K_1$ denote the dual of the complete embedding of $\rA(K)$ in $\cmin(\rA(K))$. This is a continuous affine map. Since $b \in \rA(K)^{**}$ for every $\mu \in L_1$, $\hat{b}(\mu) = (\mu \circ \iota)^{**}(\varphi(f)) = \breve{f}(\eta(\mu))$. Applying Lemma \ref{lem:breve-function} we conclude that $\hat{b}$ is lower semicontinuous. By \cite{Ped1979}*{4.3.13} the monotone closure of the self-adjoint elements of $\cmin(\rA(K))$ is contained in the norm closed subspace of the universally measurable elements. Thus $\iota^{**} \circ \varphi(S)$ consists of measurable elements.
\end{proof}

\begin{rem}
We do not know whether the span of the convex function in $\rC(K)$ is dense (see \cite{DK2019}*{Section 7.2}), and so we do not know whether the image of $\rC(K)$ under the map $\iota^{**} \circ \varphi$ is contained in the set of universally measurable elements. However, the above result is sufficient for our purposes.
\end{rem}

\section{Irreducible points and the spectral topology} \label{sec:spectral-topology}

In this section, we will introduce a notion of irreducible point in a compact nc convex set $K$. The set $\Irr(K)$ of irreducible points in $K$ contains the extreme boundary $\partial K$, and points in $\Irr(K)$ correspond to irreducible representations of the C*-algebra $\rC(K)$ of continuous nc functions on $K$. There is a natural topology on the set of irreducible points of $K$ obtained as a pullback of the topology on the C*-algebraic spectrum of $\rC(K)$. We will see that the closure of $\partial K$ in $\Irr(K)$ is a noncommutative analog of the classical Shilov boundary of a classical compact convex set.

\begin{defn}
Let $K$ be a compact nc convex set. We will say that a point $x \in K$ is {\em reducible} if it is unitarily equivalent to a direct sum $x \simeq y \oplus z$ for some $y,z \in K$. We will say that $x$ is {\em irreducible} if it is not reducible. We let $\Irr(K)$ denote the set of irreducible points in $K$.
\end{defn}

\begin{rem}
A point $x \in K_n$ is irreducible if and only if the corresponding representation $\delta_x : \rC(K) \to \cM_n$ is irreducible.
\end{rem}

Let $K$ be a compact nc convex set. Let $\Rep(\rC(K))$ denote the set of unitary equivalence classes of non-degenerate representations of $\rC(K)$ and let $\Irrep(\rC(K)) \subseteq \Rep(\rC(K))$ denote the subset of unitary equivalence classes of irreducible representations. Then $\Irrep(\rC(K))$ is the usual C*-algebraic spectrum of $\rC(K)$. For $x \in K_n$, let $[\delta_x]$ denote the unitary equivalence class of the corresponding representation $\delta_x : \rC(K) \to \cM_n$. Then the map $K \to \Rep(\rC(K)) : x \to [\delta_x]$ is surjective and the image of $\Irr(K)$ under this map is $\Irrep(\rC(K))$.

\begin{defn}
Let $K$ be a compact nc convex set. We define the {\em spectral topology} on $\Irr(K)$ to be the pullback of the hull-kernel topology on the set $\Irrep(\rC(K))$ of unitary equivalence classes of irreducible representations of $\rC(K)$ (see e.g. \cite{Dix1977}*{Chapter 3}). Specifically, a subset of $\Irr(K)$ is open in the spectral topology if it is the preimage of a set in $\Irrep(\rC(K))$ that is open in the hull-kernel topology.
\end{defn}

For a classical compact convex set $C$, the irreducible representations of the C*-algebra $\rC(C)$ of continuous functions on $C$ are precisely the point evaluations $\delta_x : \rC(C) \to \bC$ corresponding to points $x \in C$. The classical Shilov boundary for $C$ is the closure $\ol{\partial C}$ of the extreme boundary $\partial C$ in $C$, and the C*-algebra $\rC(\ol{\partial C})$ is the minimal commutative C*-algebra generated by a unital order embedding of the function system $\rA(C)$ (see \cite{DK2019}*{Section 4.1}). In particular, for $x \in C$, the corresponding point evaluation $\delta_x$ factors through $\rC(\ol{\partial C})$ if and only if $x \in \ol{\partial K}$.

The next result justifies the assertion that for a compact nc convex set $K$, the closure of $\partial K$ in $\Irr(K)$ with respect to the spectral topology is a noncommutative analog of the classical Shilov boundary.

\begin{prop} \label{prop:closure-ext-bdy-spec-top}
Let $K$ be a compact nc compact set. A point $x \in \Irr(K)_n$ belongs to the closure of $\partial K$ in the spectral topology if and only if the corresponding representation $\delta_x : \rC(K) \to \cM_n$ factors through $\cmin(\rA(K))$.  Hence $\partial K$ is closed in $\Irr(K)$ with respect to the spectral topology if and only if whenever $x \in \Irr(K)_n$ is an irreducible point with the property that the corresponding representation $\delta_x : \rC(K) \to \cM_n$ factors through $\cmin(\rA(K))$, then $x \in \partial K$. Similarly, $\partial K$ is dense in $\Irr(K)$ with respect to the spectral topology if and only if for every irreducible point $x \in \Irr(K)_n$ the corresponding representation $\delta_x : \rC(K) \to \cM_n$ factors through $\cmin(\rA(K))$.
\end{prop}

\begin{proof}
Let $\pi : \rC(K) \to \cmin(\rA(K))$ denote the canonical quotient map. Then as in the proof of \cite{DK2019}*{Theorem 6.5.1}, $\ker \pi = \cap_{y \in \partial K} \ker \delta_y$.

If $x \in \Irr(K)$ is a point that belongs to the closure of $\partial K$ in the spectral topology, then from above the corresponding representation $\delta_x$ satisfies $\ker \delta_x \supseteq \pi$, so $\delta_x$ factors through $\cmin(\rA(K))$. 

Conversely, if $x \in \Irr(K)$ has the property that $\delta_x$ factors through $\cmin(\rA(K))$, then $\ker \delta_x \supseteq \pi$, so from above $x$ belongs to the closure of $\partial K$ in the spectral topology.
\end{proof}

\begin{prop} \label{prop:not_closed_bdry}
Let $K$ be a compact nc convex set such that $\partial K$ is not closed in $\Irr(K)$ with respect to the spectral topology. Then there is $n$, a net $\{ z_i \in K_n\}$ of maximal points and a non-maximal point $z \in K_n$ such that $\lim \delta_{z_i} = \delta_z$ in the point-strong topology. 
\end{prop}

\begin{proof}
If $\partial K$ is not closed in $\Irr(K)$ with respect to the spectral topology, then there is $x \in \Irr(K)$ such that $x$ belongs to the closure of $\partial K$ with respect to the spectral topology but $x \notin \partial K$. By \cite{DK2019}*{Theorem 6.1.6}, $x$ is not a maximal point.

Let $z = \oplus_{y \in \partial K} y$. Choose cardinals $l$ and $m$ such that the amplifications $u = x^{(l)}$ and $v = z^{(m)}$ satisfy $u,v \in K_n$ for an infinite cardinal $n$. Since $x$ is not a maximal point, $u$ is not a maximal point. On the other hand, since points in $\partial K$ are maximal and direct sums of maximal points are maximal, $v$ is maximal.

By the definition of the spectral topology in terms of the hull-kernel topology,
\[
\ker \delta_u = \ker \delta_x \supseteq \cap_{y \in \partial K} \delta_y = \ker \delta_z = \ker \delta_v.
\]
Hence $\delta_u$ is weakly contained in $\delta_v$.  By \cite{Dix1977}*{3.9.8}, there is a net $\{w_i \in K_n\}$ of finite direct sums of points unitarily equivalent to $v$ such that $\delta_u = \lim \delta_{w_i}$ in the point-strong topology. It follows that $u = \lim v_i \in K$.
\end{proof}

\begin{rem}
Let $K$ be a compact nc convex set and $n$ be a nonzero cardinal. The map from $\Irr(K)_n$ to the spectrum of $C(K)$ sending a point $x \in \Irr(K)_n$ to the unitary equivalence class $[\delta_x]$ is surjective. Moreover, it follows from \cite{Dix1977}*{Theorem 3.5.8} that the map is open with respect to the relative point-weak* topology on $\Irr(K)_n$. Thus the relative spectral topology on $\Irr(K)_n$ is coarser than the relative point-weak* topology on $\Irr(K)_n$.

The preimage of the unitary equivalence class $[\delta_x]$ in $\Irr(K)_n$ is the set of unitary conjugates of $x$. In particular, a unitarily invariant set in $\Irr(K)_n$ is open in the relative spectral topology if and only if it is open in the relative point-weak* topology.
\end{rem}

\section{Noncommutative Bauer simplices} \label{sec:nc-bauer-simplices}

A simplex $C$ is said to be a Bauer simplex if the extreme boundary $\partial C$ is closed. Bauer showed that $C$ is a classical Bauer simplex if and only if it is affinely homeomorphic to the space $\rP(X)$ of probability measures on a compact Hausdorff space $X$ equipped with the weak* topology. Equivalently, $C$ is a classical Bauer simplex if and only if it is affinely homeomorphic to the state space of the unital commutative C*-algebra $\rC(X)$ of continuous functions on $X$.

In this section, we will introduce a noncommutative generalization of the notion of a classical Bauer simplex. We will show that a compact nc convex set $K$ is an nc Bauer simplex if and only if it is affinely homeomorphic to the nc state space of a unital C*-algebra. 

\begin{defn}
A compact nc convex set $K$ is a {\em nc Bauer simplex} if it is an nc Choquet simplex and the extreme boundary $\partial K$ is a closed subset of the set $\Irr(K)$ of irreducible points in $K$ with respect to the spectral topology.
\end{defn}

Proposition \ref{prop:closure-ext-bdy-spec-top} immediately implies the following characterization of nc Bauer simplices.

\begin{thm} \label{thm:characterization-nc-bauer-simplices-reps}
An nc simplex $K$ is an nc Bauer simplex if and only if the barycenter of every irreducible representation of $\rC(K)$ that factors through $\cmin(\rA(K))$ belongs to the extreme boundary $\partial K$.
\end{thm}

\begin{prop} \label{prop:nc-bauer-simplex-diagram}
Let $K$ be an nc Bauer simplex. Let $\iota : \rA(K) \to \cmin(\rA(K))$ denote the canonical unital complete order embedding and let $\iota^{**} : \rA(K)^{**} \to \cmin(\rA(K))^{**}$ denote the unique normal extension of $\iota$. Let $\pi : \rC(K) \to \cmin(\rA(K))$ denote the canonical surjective homomorphism satisfying $\pi|_{\rA(K)} = \iota$, and let $\phi : \rC(K) \to \rA(K)^{**}$ denote the unital completely positive map from Theorem \ref{thm:characterization-nc-simplex-expectation}. Then $\pi = \iota^{**} \circ \phi$, and hence the following diagram commutes:
\[
\begin{tikzcd}
\rC(K) \arrow[d, "\varphi"] \arrow[r, "\pi"] & \cmin(\rA(K)) \arrow[d,"\id",hook] \\  \rA(K)^{**} \arrow[r,"\iota^{**}"] & \cmin(\rA(K))^{**} 
\end{tikzcd}
\]
\end{prop}

\begin{proof}
Let $x \in (\partial K)_n$ be an extreme point. Since $x$ is maximal, \cite{DK2019}*{Proposition 5.2.4} implies that $\delta_x$ factors through $\cmin(\rA(K))$. Hence letting $\ol{\delta_x} : \cmin(\rA(K)) \to \cM_n$ denote the corresponding induced map and let $\ol{\delta_x}^{**} : \cmin(\rA(K))^{**} \to \cM_n$ denote the unique normal extension. Then $\ol{\delta_x}^{**} \circ \iota^{**} \circ \phi$ is a representing map for $x$. However, since $x$ is extreme, \cite{DK2019}*{Theorem 6.1.6} implies that the representation $\delta_x : \rC(K) \to \cM_n$ is the unique representing map for $x$. Hence $\delta_x = \ol{\delta_x}^{**} \circ \iota^{**} \circ \phi$. 

It follows from Theorem \ref{thm:characterization-nc-bauer-simplices-reps} that every irreducible representation of $\cmin(\rA(K))$ is equivalent to a representation of the form $\ol{\delta_x} : \cmin(\rA(K)) \to \cM_n$ for an extreme point $x \in (\partial K)_n$. Therefore, we can identify the atomic representation of $\cmin(\rA(K))$ with the representation $\pi_a : \cmin(\rA(K)) \to \B(H_a)$ with $\pi_a = \oplus_{x \in \partial K} \delta_x$.

Let $H = \oplus_n \oplus_{x \in (\partial K)_n} H_n$ and let $\sigma : \rC(K) \to \B(H)$ denote the representation defined by $\sigma = \oplus_{x \in \partial K} \delta_x$. Then letting $\pi_a : \cmin(\rA(K))^{**} \to \B(H_a)$ denote the unique normal extension of $\pi_a$, it follows from above that
\[
\sigma = \pi_a \circ j \circ \pi = \pi_a^{**} \circ \iota^{**} \circ \phi.
\]
Hence the diagram commutes if we compose with $\pi_a$.

Fix a self-adjoint element $a \in\rA(K)$  and let $f = a^2$. Every element in $\cmin(\rA(K))$ is universally measurable in the sense of Section \ref{sec:universally-measurable}, and by Lemma \ref{lem:universally-measurable} $\iota^{**}(\varphi(a^2))$ is universally measurable, as well. By \cite{Ped1979}*{Theorem 4.3.15}, the atomic representation $\pi_a$ is faithful on the universally measurable elements of $\cmin(\rA(K))^{**}$, so we conclude that $$j(\pi(a))^2 = j(\pi(a^2)) = \iota^{**}(\varphi(a^2))$$. Since $j \circ \pi|_{\rA(K))} = \iota^{**} \circ \varphi|_{\rA(K)}$ we conclude that $a$ is in the multiplicative domain of $\iota^{**} \circ \varphi$. Since $j \circ \pi$ is a *-homomorphism we obtain that $j \circ \pi = \iota^{**} \circ \varphi$, as desired.
\end{proof}

\begin{cor} \label{cor:nc-bauer-simplex-factors-maximal}
Let $K$ be an nc Bauer simplex. An nc state $\mu : \rC(K) \to \cM_n$ factors through $\cmin(\rA(K))$ if and only if it is maximal in the nc Choquet order.
\end{cor}

\begin{proof}
Let $\mu : \rC(K) \to \cM_m$ be an nc state and let $(x,\alpha) \in K_n \times \cM_{n,m}$ be a minimal representation for $\mu$. 

If $\mu$ is maximal in the nc Choquet order, then by \cite{DK2019}*{Theorem 8.3.7}, $x$ is a maximal point. Hence by \cite{DK2019}*{Proposition 5.2.4}, the representation $\delta_x : \rC(K) \to \cM_n$ factors through $\cmin(\rA(K))$, and it follows that $\mu$ factors through $\cmin(\rA(K))$. 

Conversely, if $\mu$ factors through $\cmin(\rA(K))$, then by applying Stinespring's dilation theorem to the induced map $\ol{\mu} : \cmin(\rA(K)) \to \cM_m$, we obtain a minimal representation of $\mu$ that factors through $\cmin(\rA(K))$. By the uniqueness of minimal representations (see \cite{DK2019}*{Remark 5.2.2}), it follows that the representation $\delta_x$ factors through $\cmin(\rA(K))$. 

Let $\ol{\delta_x} : \cmin(\rA(K)) \to \cM_n$ denote the induced representation and let $\ol{\delta_x}^{**} : \cmin(\rA(K))^{**} \to \cM_n$ denote the unique normal extension. Then by Proposition \ref{prop:nc-bauer-simplex-diagram},
\[
\delta_x = \ol{\delta_x} \circ \pi = \ol{\delta_x}^{**} \circ \iota^{**} \circ \phi = x^{**} \circ \phi.
\]
Hence by Corollary \ref{cor:formula-affinization-map}, $\delta_x$ is nc Choquet maximal. It follows from \cite{DK2019}*{Proposition 8.3.6} and \cite{DK2019}*{Proposition 5.2.3} that $x$ is maximal. Therefore, by \cite{DK2019}*{Theorem 8.3.7}, $\mu$ is maximal in the nc Choquet order.
\end{proof}

\begin{thm} \label{thm:characterization-nc-bauer-simplex-c-star}
Let $K$ be a compact nc convex set. Then $K$ is an nc Bauer simplex if and only if $\rA(K)$ is unitally completely order isomorphic to a unital C*-algebra, namely $\cmin(\rA(K))$.
\end{thm}

\begin{proof}
Suppose that $\rA(K)$ is unitally completely order isomorphic to a unital C*-algebra $A$. Then $\cmin(\rA(K)) = A$ and by \cite{DK2019}*{Example 6.1.5}, the extreme boundary $\partial K$ is precisely the set of irreducible representations of $A$. Hence by Theorem \ref{thm:characterization-nc-bauer-simplices-reps}, $K$ is an nc Bauer simplex.

Conversely, suppose that $K$ is an nc Bauer simplex. Let $L$ denote the nc state space of $\cmin(\rA(K))$. We will show that $K$ and $L$ are affinely homeomorphic. The result will then follow by \cite{DK2019}*{Corollary 3.2.6}.

For $x \in K_n$, let $\mu_x : \rC(K) \to \cM_n$ denote the unique nc Choquet maximal representing map for $x$. By Corollary \ref{cor:nc-bauer-simplex-factors-maximal}, $\mu_x$ factors through $\cmin(\rA(K))$. Let $\ol{\mu_x} : \cmin(\rA(K)) \to \cM_n$ denote the induced map. It follows as in the proof of Proposition \ref{prop:map-to-max-rep-measure-maximal} that the map $\ol{\theta} : K \to L$ defined by $\ol{\theta}(x) = \ol{\mu_x}$ is affine.

By Corollary \ref{cor:nc-bauer-simplex-factors-maximal}, for a map $\ol{\mu} \in L_n$, the nc state $\mu : \rC(K) \to \cM_n$ defined by $\mu = \ol{\mu} \circ \pi$ is maximal in the nc Choquet order. Hence $\mu = \mu_x$ and $\ol{\mu} = \ol{\mu_x}$, where $x \in K_n$ is the barycenter of $\mu$. Hence by the uniqueness of the nc Choquet maximal representing maps, $\ol{\theta}$ is a bijection. The inverse of $\ol{\theta}$ is the restriction map corresponding to the inclusion $\iota(\rA(K)) \subseteq \cmin(\rA(K))$, where $\iota : \rA(K) \to \cmin(\rA(K))$ denote the canonical unital complete order embedding. Since the restriction map is continuous and each $K_n$ is compact, it follows that $\ol{\theta}$ is an affine homeomorphism.
\end{proof}

\section{Noncommutative Poulsen simplices} \label{sec:nc-poulsen-simplices}

Poulsen \cite{Pou1961} constructed an example of a metrizable simplex $P$, now referred to as the Poulsen simplex, with dense extreme boundary, i.e. such that $\rC(\ol{\partial P}) = \rC(P)$. Lindenstrauss, Olsen and Sternfeld \cite{LOS1978} showed that $P$ is the unique metrizable simplex with this property. In this section we will consider noncommutative analogues of the Poulsen simplex.

Kirchberg and Wassermann \cite{KW1998}*{Theorem 17} constructed a separable nuclear C*-system $\bX$ with the property that $\cmin(\bX) = \cmax(\bX)$, which is a noncommutative analogue of the property characterizing the classical Poulsen simplex. Subsequently, Lupini \cite{Lup2018-u}*{Theorem 1.4} (see also \cites{Lup2016,Lup2018-f}) showed  that $\bX$ is the unique nuclear separable operator system satisfying this property. Since the Kirchberg-Wassermann operator system $\bX$ is a C*-system, Theorem \ref{thm:characterization-nc-simplex-c-star-system} implies that the nc state space of $\bX$ is an nc simplex.

\begin{defn} \label{defn:nc-poulsen-simplex}
A compact nc convex set $K$ is an {\em nc Poulsen simplex} if it is an nc simplex and the extreme boundary $\partial K$ is a dense subset of the set $\Irr(K)$ of irreducible points in $K$ with respect to the spectral topology.
\end{defn}

\begin{prop} \label{prop:equality-cmin-cmax}
Let $K$ be a compact nc convex set. Then the equality $\cmin(\rA(K)) = \rC(K)$ holds if and only if the extreme boundary $\partial K$ is dense in the set $\Irr(K)$ of irreducible points in $K$ with respect to the spectral topology.
\end{prop}

\begin{proof}
Let $\pi : \rC(K) \to \cmin(\rA(K))$ denote the canonical surjective homomorphism. Then by \cite{DK2019}*{}, $\ker \pi = \cap_{y \in \partial K} \ker \delta_y$. 

Suppose that $\rC(K) = \cmin(\rA(K))$. Then $\ker \pi = 0$, so for an irreducible point $x \in \Irr(K)$, the corresponding representation $\delta_x$ satisfies $\ker \delta_x \supseteq \ker \pi = \cap_{y \in \partial K} \ker \delta_y$. Then by the definition of the spectral topology in terms of the hull-kernel topology on the spectrum of $\rC(K)$, $x$ belongs to the closure of $\partial K$ with respect to the spectral topology. Hence $\partial K$ is dense in $\Irr(K)$ with respect to the spectral topology. 

Conversely, suppose that $\partial K$ is a dense subset of $\Irr(K)$ with respect to the spectral topology. Then it follows from the definition of the spectral topology in terms of the hull-kernel topology on the spectrum of $\rC(K)$ that for every point $x \in \Irr(K)$, the corresponding representation $\delta_x$ satisfies $\ker \delta_x \supseteq \cap_{y \in \partial K} \ker \delta_y = \ker \pi$. Hence $\delta_x$ factors through $\cmin(\rA(K))$, and so every irreducible representation of $\rC(K)$ factors through $\cmin(\rA(K))$, implying $\cmin(\rA(K)) = \rC(K)$. 
\end{proof}

The next result follows immediately from Definition \ref{defn:nc-poulsen-simplex} and Proposition \ref{prop:equality-cmin-cmax}.

\begin{thm} \label{thm:nc-poulsen-simplex}
An nc simplex $K$ is an nc Poulsen simplex if and only if $\cmin(\rA(K)) = \rC(K)$.
\end{thm}

The next result follows immediately from Theorem \ref{thm:nc-poulsen-simplex} and the uniqueness of the Kirchberg-Wassermann operator system from \cite{Lup2018-u}*{Theorem 1.4}. 

\begin{cor} \label{cor:unique-nuclear-nc-poulsen-simplex}
Let $\bP$ denote the nc state space of the Kirchberg-Wassermann operator system. Then $\bP$ is the unique nc Poulsen simplex with the property that $\rA(\bP)$ is separable and nuclear.
\end{cor}

Based on the uniqueness of the classical Poulsen simplex, one might suspect that there is a unique nc Poulsen simplex in the separable case. Corollary \ref{cor:unique-nuclear-nc-poulsen-simplex} shows that this is true under the additional assumption that the corresponding operator system is nuclear. However, we will see that this is false in general.

Kirchberg and Wassermann show in \cite{KW1998}*{Proposition 16} that every operator system $S$ embeds into a C*-system $X$ satisfying $\cmin(X) = \cmax(X)$. Moreover, if $S$ is separable, then $X$ can be chosen to be separable. The next result is a restatement of this result using the the dual equivalence between operator systems and compact nc convex sets (see \cite{DK2019}*{Section 3}).

\begin{prop}
Let $K$ be a compact nc convex set. Then there is an nc Poulsen simplex $L$ and a surjective affine nc map $\theta : L \to K$. If $\rA(K)$ is separable, then $L$ can be chosen so that $\rA(L)$ is separable.
\end{prop}

\begin{example}
Let $A$ be a separable non-exact C*-algebra. For example, we can take $A = \ca(\bF_2)$, the full C*-algebra of the free group (see e.g. \cite{BO2008}*{Corollary 3.7.11}). By \cite{KW1998}*{Proposition 16}, $A$ embeds into a separable C*-system $X$ satisfying $\cmin(X) = \cmax(X)$. Let $K$ denote the nc state space of $X$. Then by Theorem \ref{thm:nc-poulsen-simplex}, $K$ is an nc Poulsen simplex. Since $A$ is non-exact, $X$ is non-nuclear. In particular, $X$ is not the Kirchberg-Wassermann operator system. Hence there are multiple nc Poulsen simplices with corresponding operator systems that are separable.
\end{example}

\section{Noncommutative dynamics} \label{sec:nc-dynamics}

In this section we consider a notion of affine dynamical system over a compact nc convex set. We will show that if the compact nc convex set is an nc simplex, then the set of invariant points is an nc simplex.

\begin{defn} \label{defn:nc-dynamical-system}
An {\em nc dynamical system} is a triple $(S,\Gamma,\sigma)$ consisting of an operator system $S$, a discrete group $\Gamma$ and a group homomorphism $\sigma : \Gamma \to \operatorname{Aut}(S)$, where $\operatorname{Aut}(S)$ denotes the group of unital complete order isomorphisms of $S$. In the special case when $S$ is a C*-algebra, say $S = A$, then we will refer to the nc dynamical system $(A,\Gamma,\sigma)$ as a  {\em C*-dynamical system}. We will typically write $(S,\Gamma)$ instead of $(S,\Gamma,\sigma)$ and $s f$ instead of $\sigma_s(f)$ for $s \in \Gamma$ and $f \in S$.
\end{defn}

By the nc Kadison duality, the category of operator systems is dually equivalent to the category of compact nc convex sets. Applying this duality to an nc dynamical system results in an affine nc dynamical system. A similar notion was considered in \cite{KS2019} within the setting of matrix convexity.

\begin{defn} \label{defn:nc-affine-dynamical-system}
An \em {affine nc dynamical system} is a triple $(K,\Gamma,\kappa)$ consisting of a compact nc convex set $K$, a discrete group $\Gamma$ and a group homomorphism $\kappa : \Gamma \to \operatorname{Homeo}(K)$, where $\operatorname{Homeo}(K)$ denotes the group of affine self homeomorphisms of $K$ in the sense of Definition \ref{defn:nc_func}. We write $K^\Gamma$ for the fixed point set $K^\Gamma = \{x \in K : \kappa_s(x) = x \text{ for all } s \in \Gamma\}$. We will typically write $(K,\Gamma)$ instead of $(K,\Gamma,\kappa)$ and $s x$ instead of $\kappa_s(x)$ for $s \in \Gamma$ and $x \in K$.
\end{defn}

For an nc dynamical system of the form $(\rA(K),\Gamma)$ for a compact nc convex set $K$, the corresponding affine nc dynamical system $(K,\Gamma)$ is determined by
\[
a(s x) = (s^{-1} a)(x), \qfor s \in \Gamma,\ a \in \rA(K),\ x \in K.
\]
Conversely, for an affine nc dynamical system $(K,\Gamma)$, the corresponding nc dynamical system $(\rA(K),\Gamma)$ is determined by
\[
(s a)(x) = a(s^{-1} x), \qfor s \in \Gamma,\ a \in \rA(K),\ x \in K.
\]

The next result is an equivariant version of \cite{DK2019}*{Theorem 3.2.5}.

\begin{prop}
For a fixed group $\Gamma$, the category of nc dynamical systems over $\Gamma$ with morphisms consisting of equivariant unital complete order homomorphisms is dual to the category of affine nc dynamical systems with morphisms consisting of equivariant continuous affine nc maps.
\end{prop}

Let $K$ be a compact nc convex set and let $(\rA(K),\Gamma,\sigma)$ be an nc dynamical system. The universal property of $\rC(K)$ implies that for $s \in \Gamma$, the unital complete order isomorphism $\sigma_s$ has a unique extension to a unital complete order isomorphism of $\rC(K)$ that we continue to denote by $\sigma_s$. Thus we obtain a C*-dynamical system $(\rC(K),\Gamma,\sigma)$. As above we will typically write this as $(\rC(K),\Gamma)$, and write $sf$ instead of $\sigma_s(f)$ for $s \in \Gamma$ and $f \in \rC(K)$.

\begin{lem} \label{lem:fixed-point-set-nc-convex}
Let $(K,\Gamma)$ be an affine nc dynamical system. Then the fixed point set $K^\Gamma$ is a compact nc convex set.
\end{lem}

The following characterization of the extreme points in the set of invariant points in an affine nc dynamical system follows from \cite{DK2019}*{Theorem 6.1.6}.

\begin{prop} \label{prop:criterion-ergodic-nc-state}
Let $(K,\Gamma)$ be an affine nc dynamical system. An invariant point $x \in K^\Gamma$ is an extreme point in $K^\Gamma$ if and only if $x$ is both irreducible and maximal in $K^\Gamma$.
\end{prop}

Let $A$ be a C*-algebra and let $\phi : A \to \B(H)$ be a unital completely positive map. A {\em minimal Stinespring representation} of $\phi$ is a triple $(\pi, v, R)$ consisting of a Hilbert space $R$, a representation $\pi : A \to \B(R)$ and an isometry $v : H \to R$ such that $\phi = v^* \pi v$ and $\ol{\pi(A) v H} = R$. 

The next result is probably well known. It is a straightforward generalization of a result of Segal \cite{Seg1951}*{Theorem 5.3} for invariant states of a C*-dynamical system.

\begin{lem} \label{lem:invariant-ucp-covariant-rep}
Let $(A,\Gamma)$ be a C*-dynamical system, let $\mu : A \to \B(H)$ be an invariant unital completely positive map and let $(\pi, v, R)$ be a minimal Stinespring representation for $\mu$. Then there is a unitary representation $\rho : \Gamma \to \U(R)$, such that every vector in $v H$ is invariant for $\rho$ and the the pair $(\pi, \rho)$ is a covariant representation for $(A,\Gamma)$, meaning that
\[
\rho(s) \pi(a) \rho(s)^* = \pi(s a), \qfor s \in \Gamma,\ a \in A.
\]
If $(\pi', v', R')$ is another minimal Stinespring representation for $\mu$ with corresponding unitary representation $\rho' : \Gamma \to \U(R')$, then there is a unitary $u : R \to R'$ such that $u \pi u^* = \pi'$ and $u \rho u^* = \rho'$.
\end{lem}

\begin{proof}
For $s \in \Gamma$, define $\rho(s)$ on $\pi(A) v H$ by
\[
\rho(s) \pi(a) v \xi = \pi(s a) v \xi, \qfor a \in A,\ \xi \in H. 
\]
The minimality of the Stinespring representation $(\pi, v, H)$ and the invariance of $\mu$ implies that $\rho(s)$ is well defined and extends to a unitary on $R$. It is easy to check that the map $\rho : \Gamma \to \U(H)$ is a unitary representation of $\Gamma$ and the pair $(\pi,\rho)$ is a covariant representation for $(A,\Gamma)$.

Let $p \in \B(R)$ denote the projection onto the subspace of invariant vectors for $\rho$. The fact that $p$ is in the strong operator closed convex hull of the set $\{ \rho(s) : s \in \Gamma \}$ is a consequence of the ergodic theorem.

Let $(\pi', v', H')$ be another minimal Stinespring representation for $\mu$ with corresponding unitary representation $\rho' : \Gamma \to \U(H')$. Define $u : \pi(A) v H \to \pi'(A) v' H$ by
\[
u \pi(a) v \xi = \pi'(a) v' \xi, \qfor a \in A,\ \xi \in H.
\]
Then the minimality of $(\pi,\rho,H)$ and $(\pi',\rho',H')$ implies that $u$ extends to a unitary $u : H \to H'$ satisfying $u \pi_\mu u^* = \pi'$ and $u \rho u^* = \rho'$. 
\end{proof}

The next result is a special case of \cite{LR1967}*{Theorem 2.3}.

\begin{lem} \label{lem:compression-inv-vectors-commutes}
Let $(A,\Gamma)$ be a commutative C*-dynamical system, let $\mu : A \to \B(H)$ be an invariant unital completely positive map and let $(\pi, v, R)$ be a minimal Stinespring representation for $\mu$. Let $\rho : \Gamma \to \U(R)$ denote the unitary representation as in Lemma \ref{lem:invariant-ucp-covariant-rep} and let $p \in \B(R)$ denote the projection onto the subspace of invariant vectors for $\rho$. Then the elements in $p \pi(A) p$ commute.
\end{lem}

\begin{lem} \label{lem:ma-rep-map-equivariant}
Let $(K,\Gamma)$ be an affine nc dynamical system such that $K$ is an nc Choquet simplex. For $x \in K_m$, let $\mu_x : \rC(K) \to \cM_m$ denote the unique nc Choquet maximal representing map for $x$. Then for $s \in \Gamma$, $s \mu_x = \mu_{s x}$. In particular, if $x$ is invariant then $\mu_x$ is invariant.
\end{lem}

\begin{proof}
Let $(y,\alpha) \in K_n \times \cM_{n,m}$ be a minimal representation for $\mu_x$. Since $\mu_x$ is maximal, it follows from \cite{DK2019}*{} that $y$ is a maximal point. For $s \in \Gamma$, if $z \in K_p$ is a dilation of $s y$, then $s^{-1} z$ is a dilation of $y$, and it is clear that $z$ is a non-trivial dilation of $s y$ if and only if $s^{-1} z$ is a non-trivial dilation of $y$. Since $y$ is a maximal point, it follows that $s y$ is also a maximal point.

Since $(y,\alpha)$ is a minimal representation of $\mu_x$, the pair $(s y, \alpha)$ is a minimal representation of the map $s \mu_x$. Since $s y$ is a maximal point, \cite{DK2019}*{} implies that $s \mu_x$ is maximal. Furthermore, for $a \in \rA(K)$, 
\[
s \mu_x(a) = \mu_x(s^{-1} a) = (s^{-1} a)(x) = a(s x).
\]
Hence $s \mu_x$ has barycenter $s x$. it now follows from the uniqueness of $\mu_{s x}$ that $\mu_{s x} = s \mu_x$.
\end{proof}

Let $(K,\Gamma)$ be an affine nc dynamical system such that $K$ is an nc Choquet simplex and $x \in K^\Gamma$ be an invariant point. The next result describes the nc Choquet maximal representing maps for $x$ on $\rC(K^\Gamma)$. It is the key technical result in this section.

\begin{prop} \label{prop:max-repn-for-invariant-pt}
Let $(K,\Gamma,h)$ be an affine nc dynamical system, such that $K$ is an nc simplex. For an invariant point $x \in (K^\Gamma)_k$, let $\mu_x : \rC(K) \to \cM_k$ denote the unique nc Choquet maximal representing map for $x$ on $\rC(K)$. Let $(y,\alpha) \in K_n \times K_{n,k}$ be a minimal representation for $\mu_x$ and let $\rho : \Gamma \to \U_n$ denote the unitary representation corresponding to $\delta_y$ as in Lemma \ref{lem:invariant-ucp-covariant-rep}. Let $\xi \in \cM_{n,m}$ be an isometry with range equal to the subspace of invariant vectors for $\rho$. Let $\nu_x : \rC(K^\Gamma) \to \cM_k$ be an nc Choquet maximal representing map for $x$ on $\rC(K^\Gamma)$. Then $\nu_x = (\xi^* \alpha)^* \delta_{\xi^* y \xi} (\xi^* \alpha)$, meaning that $(\xi^* y \xi, \xi^* \alpha)$ is a representation for $\nu_x$.
\end{prop}

\begin{rem}
Note that while Proposition \ref{prop:max-repn-for-invariant-pt} asserts that the pair $(\xi^* y \xi, \xi^* \alpha)$ is a representation for $\lambda$, it does not assert that this representation is minimal. 
\end{rem}

\begin{proof}
Note that Lemma \ref{lem:ma-rep-map-equivariant} implies $\mu_x$ is invariant, so the unitary representation $\rho$ from Lemma \ref{lem:invariant-ucp-covariant-rep} is well defined.

Let $(z,\beta) \in K^\Gamma_l \times \cM_{l,k}$ be a minimal representation for $\nu_x$. Let $\mu_z : \rC(K) \to \cM_l$ denote the unique nc Choquet maximal representing map for $z$ on $\rC(K)$. Since $z$ is invariant, Lemma \ref{lem:ma-rep-map-equivariant} implies that $\mu_z$ is invariant.

Let $(w, \gamma) \in K_q \times \cM_{q,l}$ be a minimal representation for $\mu_z$ and let $\sigma : \Gamma \to \cM_q$ denote the unitary representation corresponding to $\delta_w$ as in Lemma \ref{lem:invariant-ucp-covariant-rep}. Let $\eta \in \cM_{q,p}$ be an isometry with range equal to the subspace of invariant vectors for $\sigma$.

The point $\eta^* w \eta \in K_p$ is invariant, i.e.  $\eta^* w \eta \in K^\Gamma_p$. Furthermore, since $\ran(\eta) \supseteq \ran(\gamma)$,
\[
\gamma^* \eta \eta^* w \eta \eta^* \gamma = \gamma^* w \gamma = z.
\]
In other words, the point $\eta^* w \eta$ is a dilation of $z$, so by \cite{DK2019}*{Theorem 8.5.1}, the map $\nu_z' : \rC(K^\Gamma) \to \cM_l$ defined by
\[
\nu_z' = \beta^* \gamma^* \eta \delta_{\eta^* w \eta }\eta^* \gamma \beta
\]
satisfies $\nu_z \prec_c \nu_z'$. Since $\nu_z$ is maximal in the nc Choquet order, it follows that $\nu_z' = \nu_z$.

From above, $\beta^* \gamma^* w \gamma \beta = x$. Hence the map $\mu_x' : \rC(K) \to \cM_k$ defined by $\mu_x' =  (\gamma \beta)^* \delta_w (\gamma \beta)$ has barycenter $x$ and $(w, \gamma \beta)$ is a representation of $\mu_x'$. Since $w$ is a maximal point, \cite{DK2019}*{Theorem 8.3.7} implies that $\mu_x'$ is maximal in the nc Choquet order. Hence by uniqueness, $\mu'_x = \mu_x$. It follows that there is an isometry $\zeta \in \cM_{q,n}$, such that $\zeta^* \delta_w \zeta = \delta_y$, $\zeta^* \sigma \zeta = \rho$ and $\zeta \alpha = \gamma \beta$. Note that $\eta \eta^* \zeta = \zeta \xi \xi^*$. Hence
\[
\nu_x = \beta^* \gamma^* \eta \delta_{\eta^* w \eta} \eta^* \gamma \beta = \alpha^* \xi \delta_{\xi^* y \xi} \xi^* \alpha. 
\]
\end{proof}

The next result gives a more useful description of the extreme points in the set of invariant points in an affine nc dynamical system.

\begin{prop} \label{prop:criterion-extreme-invariant}
Let $(K,\Gamma)$ be an affine nc dynamical system such that $K$ is an nc simplex. Let $x \in K^\Gamma_k$ be invariant and let $\mu_x : \rC(K) \to \cM_k$ denote the unique nc Choquet maximal representing map for $x$. Let $(y,\alpha) \in K_n \times \cM_{n,k}$ be a minimal representation for $\mu_x$ on $\rC(K)$ and let $\sigma : \Gamma \to \cM_n$ denote the unitary representation corresponding to $\delta_y$ as in Lemma \ref{lem:invariant-ucp-covariant-rep}. Let $\xi \in \cM_{m,k}$ be an isometry with range equal to the subspace of invariant vectors for $\sigma$. Then $x$ is maximal  if and only if $\ran(\alpha) = \ran(\xi)$. Thus, $x \in \partial(K^\Gamma)$ if and only if $x$ is irreducible and $\ran(\alpha) = \ran(\xi)$.
\end{prop}

\begin{proof}
Suppose $x \in \partial(K^\Gamma)$. Then by \cite{DK2019}*{Theorem 6.1.6}, $x$ is irreducible and maximal in $K^\Gamma$. Hence if we know that $x$ is maximal if and only if $\ran(\alpha) = \ran(\xi)$ we are done with the second claim.

Assume $x$ is maximal. Since $\ran(\alpha) \subseteq \ran(\xi)$, the point $\xi^* y \xi \in K^\Gamma_m$ dilates $x$. By the maximality of $x$, this dilation must be trivial, meaning that $\xi^* y \xi \simeq x \oplus x'$ for some $x' \in K^\Gamma$, where the decomposition is with respect to $\ran(\alpha)$.  In particular, this implies that
\[
\rA(K)^{**}(y) \alpha H_k \perp \ran(\xi) \ominus \ran(\alpha).
\]

Since $\mu_x$ is nc Choquet maximal, $y$ is a maximal point in $K$. By \cite{DK2019}*{Theorem 6.1.6}, $\delta_y$ is nc Choquet maximal and hence by Corollary \ref{cor:formula-affinization-map}, $\delta_y = y^{**} \circ \phi$, where $\phi : \rC(K) \to \rA(K)^{**}$ denotes the map from Theorem \ref{thm:characterization-nc-simplex-expectation}. 

Since $(y,\alpha)$ is a minimal representation, the set $\rC(K)(y) \alpha H_k$ is dense in $H_n$. But since $\delta_y = y^{**} \circ \phi$,
\[
\rC(K)(y) \alpha H_k \subseteq \rA(K)^{**}(y) \alpha H_k,
\]
and from above the latter set is orthogonal to $\ran(\xi) \ominus \ran(\alpha)$. It follows that $\ran(\xi) = \ran(\alpha)$.

Conversely, suppose that $\ran(\alpha) = \ran(\xi)$. Then we can take $\xi = \alpha$. Let $\nu_x : \rC(K^\Gamma) \to \cM_m$ be an nc Choquet maximal representing map for $x$ on $\rC(K^\Gamma)$. By Proposition \ref{prop:max-repn-for-invariant-pt}, $\nu_x = (\alpha^* \xi) \delta_{\xi^* y \xi} (\xi^* \alpha) = \delta_x$. Hence $x$ is maximal.
\end{proof}

\begin{thm} \label{thm:nc-simplex-of-invariant-points}
Let $(K,\Gamma,h)$ be an affine nc dynamical system such that $K$ is an nc simplex. Then the set
\[
K^\Gamma = \{ x \in K : s x = x \text{ for all } s \in \Gamma \}
\]
of invariant points is an nc simplex.
\end{thm}

\begin{proof}
Lemma \ref{lem:fixed-point-set-nc-convex} implies that $K^\Gamma$ is a compact nc convex set and Proposition \ref{prop:max-repn-for-invariant-pt} implies that for an invariant point $x \in K^\Gamma$, there is a uniquely determined nc Choquet maximal representing map for $x$ on $\rC(K^\Gamma)$.
\end{proof}

Let $(X,\Gamma)$ be a topological dynamical system such that $X$ is a compact Hausdorff space. We have already mentioned the classical fact that the space $\P(X)^\Gamma$ of invariant probability measures is a simplex. This is equivalent to the assertion that for an affine dynamical system $(C,\Gamma)$ such that $C$ is a Bauer simplex, the compact convex set $C^\Gamma$ of invariant points is a simplex. The next result shows that this remains true if $C$ is replaced by an arbitrary simplex. This fact may already be known, however we have been unable to find it in the literature.

\begin{cor}
Let $(C,\Gamma)$ be an affine dynamical system, such that $C$ is a simplex. Then the set $C^\Gamma$ of invariant points is a simplex.
\end{cor}

\begin{proof}
Let $K = \max(C)$, where $\max(C)$ is defined as in Section \ref{sec:basic-examples}. Since $C$ is a simplex, Theorem \ref{thm:nc-criterion-classical-simplex} implies that $K$ is an nc simplex. Hence by Theorem \ref{thm:nc-simplex-of-invariant-points}, $K^\Gamma$ is an nc simplex. In order to show that $C^\Gamma$ is a simplex, it suffices to show that $\partial (K^\Gamma) \subseteq K^\Gamma_1$. Indeed, since $K^\Gamma_1 = C^\Gamma$, the result will follow from Theorem \ref{thm:nc-criterion-classical-simplex}.

Fix $x \in \partial (K^\Gamma)_k$. Let $\mu_x : \rC(K) \to \cM_k$ denote the unique nc Choquet maximal representing map for $x$ on $\rC(K)$. Let $(y,\alpha) \in K_n \times \cM_{n,k}$ be a minimal representation for $\mu_x$ and let $\rho : \Gamma \to \U_n$ denote the unitary representation corresponding to $\delta_y$ as in Lemma \ref{lem:invariant-ucp-covariant-rep}. Since $x$ is an extreme point in $K^\Gamma$, Proposition \ref{prop:criterion-extreme-invariant} implies that the subspace of invariant vectors for $\rho$ is precisely the range of $\alpha$.

Since $\mu_x$ is nc Choquet maximal on $\rC(K)$, \cite{DK2019}*{Theorem 8.3.7} implies that $y$ is maximal in $K$. Hence by \cite{DK2019}*{Proposition 5.2.4}, $\delta_y$ factors through $\cmin(\rA(K)) = \rC(\ol{\partial C})$. In particular, the range of $\delta_y$ is a commutative C*-algebra. Hence by Lemma \ref{lem:compression-inv-vectors-commutes}, the elements in $\alpha^* \rC(K)(y) \alpha$ commute. Since
\[
\rA(K^\Gamma)(x) = \alpha^* \rA(K)(x) \alpha \subseteq \alpha^* \rC(K)(y) \alpha,
\]
it follows that the elements in $\rA(K^\Gamma)(x)$ commute. Since $x$ is irreducible, we conclude that $x \in K^\Gamma_1$.
\end{proof}

\section{Noncommutative ergodic decomposition theorem} \label{sec:nc-ergodic-theory}

Let $(X,\Gamma)$ be a topological dynamical system where $X$ is a compact Hausdorff space. The classical ergodic decomposition theorem is a consequence of the Choquet-Bishop-de Leeuw integral representation theorem, which asserts that for an invariant probability measure $\mu \in \rP(X)^\Gamma$, there is a unique probability measure $\nu \in \rP(\rP(X)^\Gamma)$ with barycenter $\mu$ that is maximal in the Choquet order. In particular, $\mu$ is supported on the set $\partial(\rP(X)^\Gamma)$ of ergodic probability measures and
\[
\mu = \int_{\rP(X)^\Gamma} \lambda\, d\nu(\lambda).
\]

In this section we will apply the results in Section \ref{sec:nc-dynamics} to C*-dynamical systems. We will introduce a notion of ergodic nc state and and obtain a noncommutative analogue of the classical ergodic decomposition theorem.

\begin{defn}
Let $(A,\Gamma)$ be a C*-dynamical system and let $K$ denote the nc state space of $A$. We will say that an nc state $\mu \in K$ is {\em ergodic} if it is an extreme point of $K^\Gamma$. 
\end{defn}

Since we will be primarily interested in C*-dynamical systems in Section \ref{sec:nc-ergodic-theory} and Section \ref{sec:kazhdan}, we now take the time to discuss how the results in Section \ref{sec:nc-dynamics} specialize to that setting.

Let $A$ be a C*-algebra with nc state space $K$ so that $A$ is unitally completely order isomorphic to $\rA(K)$. Let $\mu : A \to \cM_m$ be an nc state and let $\hat{\mu} : \rC(K) \to \cM_m$ denote the nc state on $\rC(K)$ obtained via composition with the canonical surjective homomorphism from $\rC(K)$ onto $\cmin(\rA(K)) = A$.

Let $(\pi,v,H)$ be a minimal Stinespring representation for $\mu$. We can assume that $H = H_n$ so that $\pi$ corresponds to a point $x \in K_n$ and $v = \alpha$ for an isometry $\alpha \in \cM_{n,m}$. By \cite{DK2019}*{Example 6.1.5}, the point $x$ is maximal, so by \cite{DK2019}*{Proposition 5.2.3}, $\delta_x$ is the unique representing map for $x$.

As above, let $\hat{\pi} : \rC(K) \to \cM_n$ denote the representation of $\rC(K)$ obtained via composition with the canonical surjective homomorphism from $\rC(K)$ onto $A$. Then $\hat{\pi}$ is a representing map for $x$, so from above $\hat{\pi} = \delta_x$. In particular, this implies that the pair $(x,\alpha) \in K_n \times \cM_{n,m}$ is a representation for $\hat{\mu}$.

By the minimality of the Stinespring representation $(\pi,\alpha,H)$, the pair $(x,\alpha)$ is a minimal representation for $\hat{\mu}$. Therefore, by the uniqueness of minimal representations (see \cite{DK2019}*{Remark 5.2.2}), we see that there is a bijective correspondence between minimal Stinespring representations of $\mu$ and minimal representations of $\mu$ on $\rC(K)$.

The following characterization of the ergodic nc states in a C*-dynamical system follows from the above discussion and Proposition \ref{prop:criterion-extreme-invariant}. 

\begin{prop}
Let $(A,\Gamma)$ be a C*-dynamical system. Let $\mu : A \to \cM_n$ be an invariant nc state and let $(\pi, v, H)$ be a minimal Stinespring representation for $\mu$. Let $\sigma : \Gamma \to \U(H)$ denote the unitary representation corresponding to $\pi$ as in Lemma \ref{lem:invariant-ucp-covariant-rep}. Then $\mu$ is ergodic if and only if it is irreducible and $\ran(v)$ is precisely the set of invariant vectors for $\sigma$.
\end{prop}

\begin{rem}
Let $(X,\Gamma)$ be a topological dynamical system where $X$ is a compact Hausdorff space and consider the corresponding C*-dynamical system $(\rC(X), \Gamma)$. By the Riesz-Markov-Kakutani representation theorem, we can identify an invariant probability measure $\mu$ on $X$ with an invariant state on $\rC(X)$. 

Let $\pi : \rC(X) \to \B(\rL^2(X,\mu))$ denote the representation defined by $\pi_\mu(f) = M_f$ for $f \in \rC(X)$, where $M_f \in \B(\rL^2(X,\mu))$ denotes the multiplication operator corresponding to $f$. Then $(\pi,1_X,\rL^2(X,\mu))$ is the GNS representation for $\mu$. The corresponding unitary representation $\sigma : \Gamma \to \U(\rL^2(X,\mu))$ is defined by $\sigma(s) h = h \circ s^{-1}$ for $s \in \Gamma$ and $h \in \rL^2(X,\mu)$. 

Proposition \ref{prop:criterion-ergodic-nc-state} asserts that the probability measure $\mu$ is ergodic if and only if the only functions in $\rL^2(X,\mu)$ that are invariant for $\sigma$ are the constant functions. This is equivalent to the classical characterization of ergodic probability measures on $X$. 
\end{rem}

For a compact nc convex set $K$, $\rB^\infty(K)$ denotes the C*-algebra of bounded Baire nc functions on $K$ \cite{DK2019}*{Definition 9.2.1}. This is the Baire-Pedersen envelope of the C*-algebra $\rC(K)$, i.e. the C*-algebra obtained as the monotone sequential closure of $\rC(K)$ in its universal representation. Since $\rC(K) \subseteq \rB^\infty(K) \subseteq \rC(K)^{**}$, elements in $\rB^\infty(K)$ are, in particular, bounded nc functions.

An nc state $\mu : \rC(K) \to \cM_n$ is {\em supported} on the extreme boundary $\partial K$, if $\mu(f) = 0$ for every bounded Baire nc function $f \in \rB^\infty(K)$ satisfying $f(x) = 0$, for every $x \in \partial K$.

The next result is a noncommutative ergodic decomposition theorem with uniqueness. Note that the C*-dynamical system is not required to be separable.

\begin{thm}[Noncommutative ergodic decomposition] \label{thm:nc-ergodic-decomposition} \strut \\
Let $(A,\Gamma)$ be a C*-dynamical system and let $K$ denote the nc state space of $A$. For an invariant nc state $\mu \in K^\Gamma$, there is a unique nc state $\nu\colon\rC(K^\Gamma) \to \cM_n$, that is maximal in the nc Choquet order and represents $\mu$. Moreover, $\nu$ is supported on the space $\partial (K^\Gamma)$ of ergodic nc states.
\end{thm}

\begin{proof}
The existence of an nc Choquet maximal representing map $\nu : \rC(K) \to \cM_n$ that is supported on $\partial (K^\Gamma)$ follows from the noncommutative Choquet-Bishop-de Leeuw theorem \cite{DK2019}*{Theorem 9.2.3}. By Theorem \ref{thm:nc-simplex-of-invariant-points}, $K^\Gamma$ is an nc simplex, and hence $\nu$ is unique.
\end{proof}

In the separable case, we obtain an integral decomposition of invariant nc states in terms of ergodic nc states that more closely resembles the classical ergodic decomposition theorem. Before stating the result, we will briefly review the notion of integration against a noncommutative probability measure introduced in \cite{DK2019}*{Section 10}.

Let $K$ be a compact nc convex set such that $\rA(K)$ is separable. For $n \leq \aleph_0$, an {\em $\cM_n$-valued finite nc measure} on $K$ is a sequence $\lambda = (\lambda_m)_{m \leq \aleph_0}$ such that each $\lambda_m$ is a $\ncpmaps(\cM_m,\cM_n)$-valued Borel measure and the sum
\[
\sum_{m \leq \aleph_0} \lambda_m(K_m)(1_m) \in \cM_n
\]
is weak*-convergent. For $E \in \bor(K)$, the value $\lambda(E)$ is defined by
\[
\lambda(E) = \sum_{m \leq \aleph_0} \lambda_m(E_m).
\]
If $\lambda(K) = 1$, then $\lambda$ is said to be an {\em $\cM_n$-valued nc probability measure}.

Since each component $K_m$ of $K$ is metrizable, it follows from \cite{DK2019}*{Lemma 10.3.1} that each component $(\partial K)_m$ of the extreme boundary $\partial K$ is Borel. We say $\lambda$ is supported on $\partial K$ if
\[
\lambda_m(K_m \setminus (\partial K)_m) = 0_{\cM_m,\cM_n}, \qfor m \leq \aleph_0.
\]

In order to make sense of the integral of an nc function in $\rB^\infty(K)$ against $\lambda$, we require $\lambda$ to be {\em admissible}, meaning that each $\lambda_m$ is absolutely continuous with respect to a scalar-valued measure on $K_m$. 

For $f \in \rB^\infty(K)$, the restriction $f|_{K_m} : K_m \to \cM_m$ is a bounded and measurable $\cM_m$-valued function on $K_m$. Since $\cM_m$ is metrizable, the relative weak* topology on a bounded subset of $\cM_m$ containing the range of $f|_{K_m}$ is metrizable. It was shown in \cite{DK2019}*{Section 10.2} that $f|_{K_m}$ can be approximated uniformly with respect to this metric by simple functions, for which integration against $\lambda_m$ is defined in the obvious way. The fact that $\lambda_m$ is absolutely continuous with respect to a scalar-valued measure on $K_m$ ensures that the integral of $f|_{K_m}$ can be defined as the limit of integrals of approximating simple functions. 

For $f \in \rB^\infty(K)$, the integral of $f$ with respect to $\lambda$ is defined by
\[
\int_K f\, d \lambda := \sum_{m \leq \aleph_0} \int_{K_m} f\, d\lambda_m.
\]
Combining \cite{DK2019}*{Theorem 10.3.10} with Theorem \ref{thm:nc-ergodic-decomposition}, we obtain the following integral representation theorem.

\begin{thm}
Let $(A,\Gamma)$ be a separable C*-dynamical system and let $K$ denote the nc state space of $A$. For an invariant nc state $\mu \in K^\Gamma$ there is an admissible nc probability measure $\lambda$ on $K$ that represents $\mu$ and is supported on the set $\partial (K^\Gamma)$ of ergodic nc states on $A$, meaning that
\[
\mu(a) = \int_K a\ d\lambda, \qfor a \in \A.
\]
\end{thm}

\section{Kazhdan's property (T)} \label{sec:kazhdan}

Glasner and Weiss \cite{GW1997}*{Theorem 1} proved an important dynamical characterization of property (T). Specifically, they proved that a locally compact second countable group $G$ has property (T) if and only if for every topological dynamical system $(X,G)$, the space $\rP(X)^G$ of invariant probability measures on $X$ is a Bauer simplex. In this section we will establish a noncommutative extension of their result for discrete groups.

For a discrete group $\Gamma$, a unitary representation $\rho : \Gamma \to \U(H)$ is said to have {\em almost invariant vectors} if there is a sequence of unit vectors $\{ \xi_i \in H \}$ such that $\lim \|\rho(s)\xi_i - \xi_i \| = 0$ for all $s \in \Gamma$. 

\begin{defn}[Kazhdan's property (T)]
A discrete group $\Gamma$ is said to have {\em Kazhdan's property (T)}, or briefly {\em property (T)}, if every unitary representation of $\Gamma$ with almost invariant vectors has a nonzero invariant vector.
\end{defn}

The next theorem is a noncommutative extension of \cite{GW1997}*{Theorem 1}.

\begin{thm} \label{thm:char-prop-T}
Let $(A,\Gamma)$ be a C*-dynamical system with nc state space $K$. If $\Gamma$ has Kazhdan's property (T), then the set $K^\Gamma$ of invariant nc states in $K$ is an nc Bauer simplex.
\end{thm}

\begin{proof}
Suppose that $\Gamma$ has property (T) and suppose for the sake of contradiction that $K^\Gamma$ is not an nc Bauer simplex. By Proposition \ref{prop:not_closed_bdry}, there is a net of maximal nc states $\{\mu_i \in K^\Gamma_m\}$ and a non-maximal nc state $\mu \in K^\Gamma_m$ such that $\lim \mu_i = \mu$ in the point-strong topology.

Let $(\pi_i,\beta_i) \in K_{n_i} \times \cM_{n_i,m}$ be a minimal representation for each $\mu_i$ and let $\sigma_i : \Gamma \to \U_{n_i}$ denote the corresponding unitary representation as in Lemma \ref{lem:invariant-ucp-covariant-rep}. Similarly, let $(\pi,\beta) \in K_n \times \cM_{n,m}$ be a minimal representation for $\mu$ and let $\sigma : \Gamma \to \U_n$ denote the corresponding unitary representation, as in Lemma \ref{lem:invariant-ucp-covariant-rep}.

By the discussion at the beginning of Section \ref{sec:nc-ergodic-theory}, each $(\pi_i,\beta_i,H_{n_i})$ is a minimal Stinespring representation for $\mu_i$ and $(\pi,\beta,H_n)$ is a minimal Stinespring representation for $\mu$. 

Since $\mu_i$ is maximal in $K^\Gamma$, Proposition \ref{prop:criterion-extreme-invariant} implies that the subspace of invariant vectors for $\sigma_i$ is precisely $\ran(\beta_i)$. Hence $\sigma_i \simeq 1_{\ran(\beta_i)} \oplus \sigma'_i$ for a unitary representation $\sigma'_i : \Gamma \to \ran(\beta_i)^\perp$ that has no nonzero invariant vectors, where the decomposition is taken with respect to $\ran(\beta_i)$.

Let $H = \oplus \ran(\beta_i)^\perp$ and consider the unitary representation $\rho : \Gamma \to \U(H)$ defined by $\rho = \oplus \sigma'_i$. Since each $\sigma_i'$ has no nonzero invariant vectors, $\rho$ has no nonzero invariant vectors. We will show that $\rho$ has almost invariant vectors. Since $\Gamma$ has property (T), this will imply that $\rho$ has an invariant unit vector, giving a contradiction. It suffices to show that for every $\epsilon > 0$ there is a unit vector $\omega \in H$ such that $\|\rho(s) \omega - \omega \| < \epsilon$ for all $s \in \Gamma$.

Fix $\epsilon > 0$ and choose $\epsilon_0 \in (0,1)$ such that $3\epsilon_0(1-\epsilon_0^2)^{-1/2} < \epsilon$.
Since $\mu$ is not maximal in $K^\Gamma$, there is a unit vector $\zeta \in \ran(\beta)^\perp$ that is invariant for $\sigma$. Also, since the representation $(\pi,\beta)$ is minimal, there is $\xi \in H_m$ and $a \in A$ such that $\|\pi(a) \beta \xi\| = 1$ and $\|\pi(a)\beta \xi - \zeta\| < \epsilon_0/2$. Then the invariance of $\zeta$ for $\sigma$ implies that for $s \in \Gamma$,
\begin{align*}
\|\sigma(s) \pi(a) \beta \xi &- \pi(a) \beta \xi \| \\
&\leq \|\sigma(s) \pi(a) \beta \xi - \sigma(s) \zeta \| + \|\sigma(s) \zeta - \pi(a) \beta \xi\| \\
&= 2\|\pi(a) \beta \xi - \zeta\| \\
&< \epsilon_0.
\end{align*}

Since $\lim \mu_i = \mu$ in the point-strong topology,
\[
\lim \|\pi_i(a) \beta_i \xi\|^2 = \lim \langle \mu_i(a^* a) \xi, \xi \rangle = \langle \mu(a^* a) \xi, \xi \rangle = \|\pi(a) \beta \xi\| = 1
\]
and
\begin{align*}
\lim &\| \beta_i^* \pi_i(a) \beta_i \xi \| = \lim \| \mu_i(a) \xi \| = \| \mu(a) \xi \| = \| \beta^* \pi(a) \beta \xi \| \\
&\quad = \| \beta^* \pi(a) \beta \xi - \beta^* \zeta \| \leq \|\pi(a) \beta \xi - \zeta \| < \epsilon_0.
\end{align*}
Hence $1 - \epsilon_0^2 < \lim \|(1-\beta_i \beta_i^*) \pi_i(a) \beta_i \xi \|^2 \leq 1$.
Similarly, it follows from the construction of the unitaries in Lemma~\ref{lem:invariant-ucp-covariant-rep} that for $s \in \Gamma$,
\begin{align*}
\lim \| \sigma_i(s) \pi_i(a) \beta_i \xi - \pi_i(a) \beta_i \xi \|^2 &= 2 - 2 \lim \operatorname{Re}\langle \sigma_i(s) \pi_i(a) \beta_i \xi, \pi_i(a) \beta_i \xi \rangle \\
&= 2 - 2 \lim \operatorname{Re}\langle \mu_i(a^* (s a)) \xi, \xi \rangle \\
&= 2 - 2 \operatorname{Re}\langle \mu(a^* (s a)) \xi, \xi \rangle \\
&= 2 - 2 \operatorname{Re}\langle \sigma(s) \pi(a) \beta \xi, \pi(a) \beta \xi \rangle \\
&< \epsilon_0^2.
\end{align*}

For each $i$, let
\[
\omega_i = \frac{1}{\|(1-\beta_i \beta_i^*) \pi_i(a) \beta_i \xi\|} (1-\beta_i \beta_i^*) \pi_i(a) \beta_i \xi \in \ran(\beta_i)^\perp.
\]
Then $\omega_i$ is a unit vector and it follows from the above inequalities that
\begin{align*}
\lim \|\sigma_i(s) \omega_i &- \omega_i\| \\
&\leq \lim \frac{1}{\|(1-\beta_i \beta_i^*) \pi_i(a) \beta_i \xi\|} \|\sigma_i(s) \pi_i(a) \beta_i \xi - \pi_i(a) \beta_i \xi \| \\
&\quad + \lim \frac{2}{\|(1-\beta_i \beta_i^*) \pi_i(a) \beta_i \xi\|} \|\beta_i^* \pi_i(a) \beta_i \xi\| \\
&< \frac{3 \epsilon_0}{\sqrt{1-\epsilon_0^2}} \\
&< \epsilon
\end{align*}
Therefore, identifying each $\omega_i$ with a vector in $H$, we conclude that for sufficiently large $i$, $\|\rho(s)\omega_i - \omega_i\| < \epsilon$.
\end{proof}

By Bauer's theorem, a compact convex set is a Bauer simplex if and only if it is affinely homeomorphic to the state space of a commutative C*-algebra. Hence the result of Glasner and Weiss, specialized to discrete groups, is equivalent to the assertion that a discrete group $\Gamma$ has property (T) if and only if for every commutative C*-dynamical system $(\rC(X),\Gamma)$ with state space $K_1$, the space $K_1^\Gamma$ of invariant states is the state space of a unital commutative C*-algebra. The following corollary is a noncommutative extension of this result.

\begin{cor}
A discrete group $\Gamma$ has Kazhdan's property (T) if and only if for every C*-dynamical system $(A,\Gamma)$ with state space $K_1$, the space $K_1^\Gamma$ of invariant states is affinely homeomorphic to the state space of a C*-algebra.
\end{cor}

\begin{proof}
Let $K$ denote the nc state space of $K$. If $\Gamma$ has property (T), then Theorem \ref{thm:char-prop-T} implies that $K^\Gamma$ is an nc Bauer simplex. It follows from Theorem \ref{thm:characterization-nc-bauer-simplex-c-star} that $K_1$ is affinely homeomorphic to the state space of a C*-algebra. 

Conversely, if $\Gamma$ does not have property (T), then by \cite{GW1997}*{Theorem 2}, for the Bernoulli shift $\{\{0,1\}^\Gamma, \Gamma\}$, the space $\rP(\{0,1\}^\Gamma)^\Gamma$ of invariant probability measures is a Poulsen simplex. The result now follows from the fact that the state space of a unital commutative C*-algebra is a Bauer simplex and the fact that the state space of a unital noncommutative C*-algebra is never a simplex.
\end{proof}


\end{document}